\documentclass[psamsfonts]{amsart}

\usepackage{setspace}
\usepackage{tipa}
\usepackage{mathrsfs}
\usepackage{amsfonts}
\usepackage{amssymb,amsfonts}
\usepackage[all,arc]{xy}
\usepackage{enumerate}
\usepackage{mathrsfs}
\usepackage{yhmath}
\usepackage{graphicx}
\def\dis{\displaystyle}

\newtheorem{thm}{Theorem}[section]
\newtheorem{cor}[thm]{Corollary}
\newtheorem{prop}[thm]{Proposition}
\newtheorem{lem}[thm]{Lemma}

\theoremstyle{definition}
\newtheorem{defn}[thm]{Definition}

\theoremstyle{remark}
\newtheorem{rem}[thm]{Remark}

\numberwithin{equation}{section}

\begin{document}

\date{}
\date{}
\title{Mean curvature flow with driving force on fixed extreme points}

\author{Longjie ZHANG}

\date{December, 2015\\Corresponding author University:Graduate School of Mathematical Sciences, The University of Tokyo. Address:3-8-1 Komaba Meguro-ku Tokyo 153-8914, Japan. Email:zhanglj@ms.u-tokyo.ac.jp, zhanglj919@gmail.com}

\maketitle

\begin{minipage}{140mm}

{{\bf Abstract:} In this paper, we consider the mean curvature flow with driving force on fixed extreme points in the plane. We give a general local existence and uniqueness result of this problem with $C^2$ initial curve. For a special family of initial curves, we classify the solutions into three  categories. Moreover, in each category, the asymptotic behavior is given.

}

{\bf 2010MSC:} 35A01, 35A02, 35K55, 53C44.

\end{minipage}

$$$$

\section{Introduction}\large

In this paper, we consider the mean curvature with driving force on fixed extreme points given by
\begin{equation}\label{eq:meancur}
V=-\kappa+A,\ \text{on}\ \Gamma(t),
\end{equation}
\begin{equation}\label{eq:initial}
\Gamma(0)=\Gamma_0.
\end{equation}
Here $V$ denotes the upward normal velocity(the definition of ``upward'' is given by Remark \ref{rem:direction}). The sign $\kappa$ is chosen such that the problem is parabolic. $A$ is a positive constant.

If we use the arc length parameter $s$ to represent $\Gamma(t)$ by $\Gamma(t)=\{F(s,t)\in\mathbb{R}^2\mid 0\leq s\leq L(t)\}$, the equation (\ref{eq:meancur}), (\ref{eq:initial}) can be written as
\begin{equation}\label{eq:meancurpara}
\frac{d}{d t}F(s,t)=\kappa N-AN,\ 0<s<L(t),
\end{equation}
\begin{equation}\label{eq:initialpara}
F(s,0)=F_0(s),\ 0\leq s\leq L_0.
\end{equation}
Here $\Gamma_0=\{F_0(s)\in\mathbb{R}^2\mid 0\leq s\leq L_0\}$; $N$ denotes the unit downward normal vector(the definition of ``downward'' is given by Remark \ref{rem:direction}) and $L(t)$ denotes the length of $\Gamma(t)$. And the notation $\frac{d}{d t}F(s,t)$ denotes the derivative of $t$ by fixing $s$. Noting the assumption that the sign of $\kappa$ is chosen such that the problem (\ref{eq:meancur}) is parabolic, combining Frenet formulas, there holds
$$
\kappa N=\frac{\partial^2}{\partial s^2}F(s,t).
$$ 
Here we give the fixed extreme point boundary condition.
\begin{equation}\label{eq:fixbound}
F(0,t)=P,\ F(L(t),t)=Q,
\end{equation}
where $P$, $Q$ are two different fixed points in $\mathbb{R}^2$.

{\bf Main results} Here we give our main theorems.

\begin{thm}\label{thm:exist}
Assume that $F_0(s)\in C^2([0,L_0]\rightarrow\mathbb{R}^2)$. Then there exist $T>0$ and unique $F(s,t)$ such that $F(s,t)$ satisfies (\ref{eq:meancurpara}), (\ref{eq:fixbound}), for $0<t<T$ and initial condition (\ref{eq:initialpara}).
\end{thm}

Assume that $P=(-a,0)$, $Q=(a,0)$, where $0<a\leq 1/A$. Before giving the three categories result, we introduce two equilibrium solutions of (\ref{eq:meancur}) with boundary condition (\ref{eq:fixbound}). Denote 
$$
\Gamma_*=\{(x,y)\in\mathbb{R}^2\mid y=\sqrt{1/A^2-x^2}-\sqrt{1/A^2-a^2}, -a\leq x\leq a\}
$$
and
$$
\Gamma^*=\partial B_{\frac{1}{A}}\big((0,\sqrt{1/A^2-a^2})\big)\setminus\{(x,-y)\in\mathbb{R}^2\mid (x,y)\in\Gamma_*\}.
$$

\begin{figure}[htbp]
	\begin{center}
            \includegraphics[height=6.0cm]{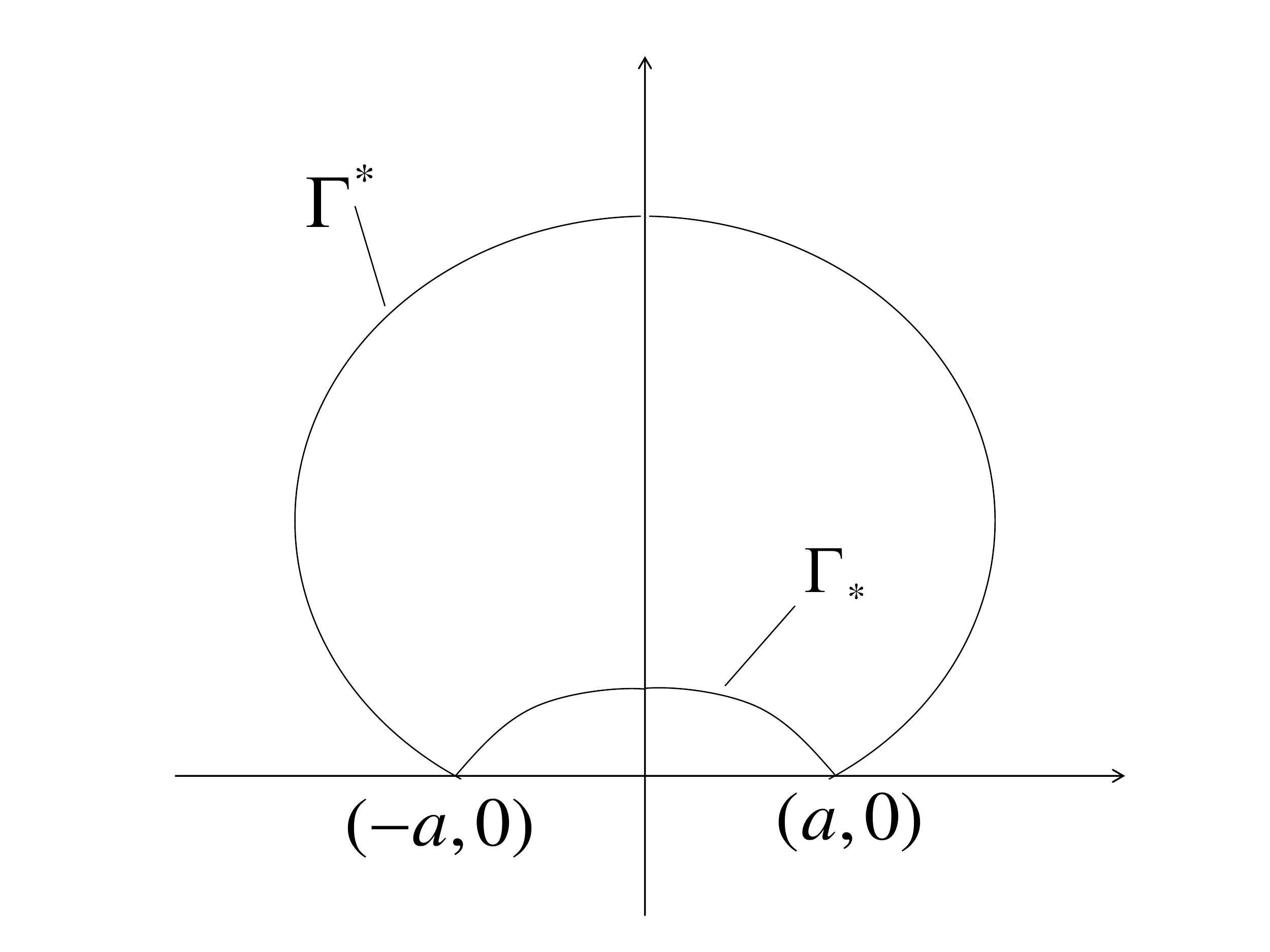}
		\vskip 0pt
		\caption{Equilibrium solutions of (\ref{eq:meancur})}
        \label{fig:stationaryso}
	\end{center}
\end{figure}

Obviously, on $\Gamma_*$ and $\Gamma^*$, there holds $\kappa=A$ and the fixed boundary condition. Here we give the three categories theorem. In the following theorem we consider a family of initial curves given by
$$
\Gamma_{\sigma}=\{(x,y)\in\mathbb{R}^2\mid y=\sigma\varphi(x), -a\leq x\leq a\}.
$$
Here $\varphi$ is even, $\varphi\in C^2\big([-a,a]\big)$ and $\varphi^{\prime\prime}(x)\leq0$, $-a<x<a$. And assume that for all $\sigma\in\mathbb{R}$, $\Gamma_{\sigma}$ intersects $\Gamma^*$ at most fourth(including the extreme points). Denote $\Gamma_{\sigma}(t)$ being the solution with $\Gamma_{\sigma}(0)=\Gamma_{\sigma}$.
\begin{thm}\label{thm:category}
There exists $\sigma^*>0$ such that 

(1). for $\sigma>\sigma^*$, there exists $T_{\sigma}^*<T_{\sigma}$ such that $\Gamma_{\sigma}(t)\succ \Gamma^*$, $T_{\sigma}^*<t<T_{\sigma}$;

(2). for $\sigma=\sigma^*$, $T_{\sigma}=\infty$ and $\Gamma_{\sigma}(t)\rightarrow \Gamma^*$ in $C^1$, as $t\rightarrow \infty$;

(3). for $\sigma<\sigma^*$, $T_{\sigma}=\infty$ and $\Gamma_{\sigma}(t)\rightarrow \Gamma_*$ in $C^1$, as $t\rightarrow \infty$.

Where $T_{\sigma}$ denotes the maximal existence time of $\Gamma(t)$.
\end{thm}
 The notation ``$\succ$'' can be seen as an order. The precise definition is given in Section 2. We will interpret the sense of $C^1$ convergence in Definition \ref{def:C1close}.

{\bf Main method.} Theorem \ref{thm:exist} can be easily proved by transport map. The transport map is first used by \cite{A} to consider the mean curvature flow under the non-graph condition. For the three categories result, we use the intersection number principle to classify the type of the solutions in Lemma \ref{lem:sigulartime}. Since $\Gamma_{\sigma}$ intersects $\Gamma^*$ at most fourth, the intersection number between $\Gamma_{\sigma}(t)$ and $\Gamma^*$ can only be two or four. In Lemma \ref{lem:sigulartime}, one of the following three conditions can hold:

(1). The curve $\Gamma_{\sigma}(t)$ intersects $\Gamma^*$ twice and $\Gamma_{\sigma}(t)\succ \Gamma^*$ eventually;

(2). The curve $\Gamma_{\sigma}(t)$ intersects $\Gamma^*$ fourth for every $t>0$. 

(3).  The curve $\Gamma_{\sigma}(t)$ intersects $\Gamma^*$ twice and $\Gamma^*\succ\Gamma_{\sigma}(t)$ eventually. 

Seeing future, under the condition (2) above, $\Gamma_{\sigma}(t)\rightarrow \Gamma^*$ in $C^1$, as $t\rightarrow \infty$; under the condition (3) above, $\Gamma_{\sigma}(t)\rightarrow \Gamma_*$ in $C^1$, as $t\rightarrow \infty$. In this paper, we prove the asymptotic behavior by using Lyapunov function introduced in Section 5. 

{\bf A short review for mean curvature flow.} For the classical mean curvature flow: $A=0$ in (\ref{eq:meancur}), there are many results. Concerning this problem, Huisken \cite{H} shows that any solution that starts out as a convex, smooth, compact surface remains so until it shrinks to a "round point" and its asymptotic shape is a sphere just before it disappears. He proves this result for hypersurfaces of $\mathbb{R}^{n+1}$ with $n\geq2$, but Gage and Hamilton \cite{GH} show that it still holds when $n=1$, the curves in the plane. Gage and Hamilton also show that embedded curve remains embedded, i.e. the curve will not intersect itself. Grayson \cite{Gr} proves the remarkable fact that such family must become convex eventually. Thus, any embedded curve in the plane will shrink to "round point" under curve shortening flow. 

For fixed extreme point problem, Forcadel, Imbert and Monneau \cite{FIM} consider a family of half lines evolves by (\ref{eq:meancur}) and one extreme point is fixed at the origin. Precisely, the family of curves given by polar coordinate,
$$
\left\{
\begin{array}{lcl}
x=\rho\cos\theta(\rho,t),\\
y=\rho\sin\theta(\rho,t),
\end{array}
\right.
$$
for $0\leq \rho<\infty$. Therefore, $\theta(\rho,t)$ satisfies
\begin{equation}\label{eq:vistheta}
\rho\theta_t=A\sqrt{1+\rho^2\theta_{\rho}^2}+\theta_{\rho}\big(\frac{2+\rho^2\theta_{\rho}^2}{1+\rho^2\theta_{\rho}^2}\big)+\frac{\rho\theta_{\rho\rho}}{1+\rho^2\theta_{\rho}^2},\ t>0,\rho>0.
\end{equation}
Obviously, this problem is singular near $\rho=0$. They consider the solution of (\ref{eq:vistheta}) in viscosity sense. Since near the fixed extreme point mean curvature flow has singularity by using polar coordinate, there are some papers considering this problem by digging a hole. For example,
 Giga, Ishimura and Kohsaka \cite{GIK} consider anisotropic curvature flow equation with driving force in the ring domain $r<\rho<R$. At the boundary, the family of the curves is imposed being perpendicular to the boundary, seeing Figure \ref{fig:dighole}.
\begin{figure}[htbp]
	\begin{center}
            \includegraphics[height=6.0cm]{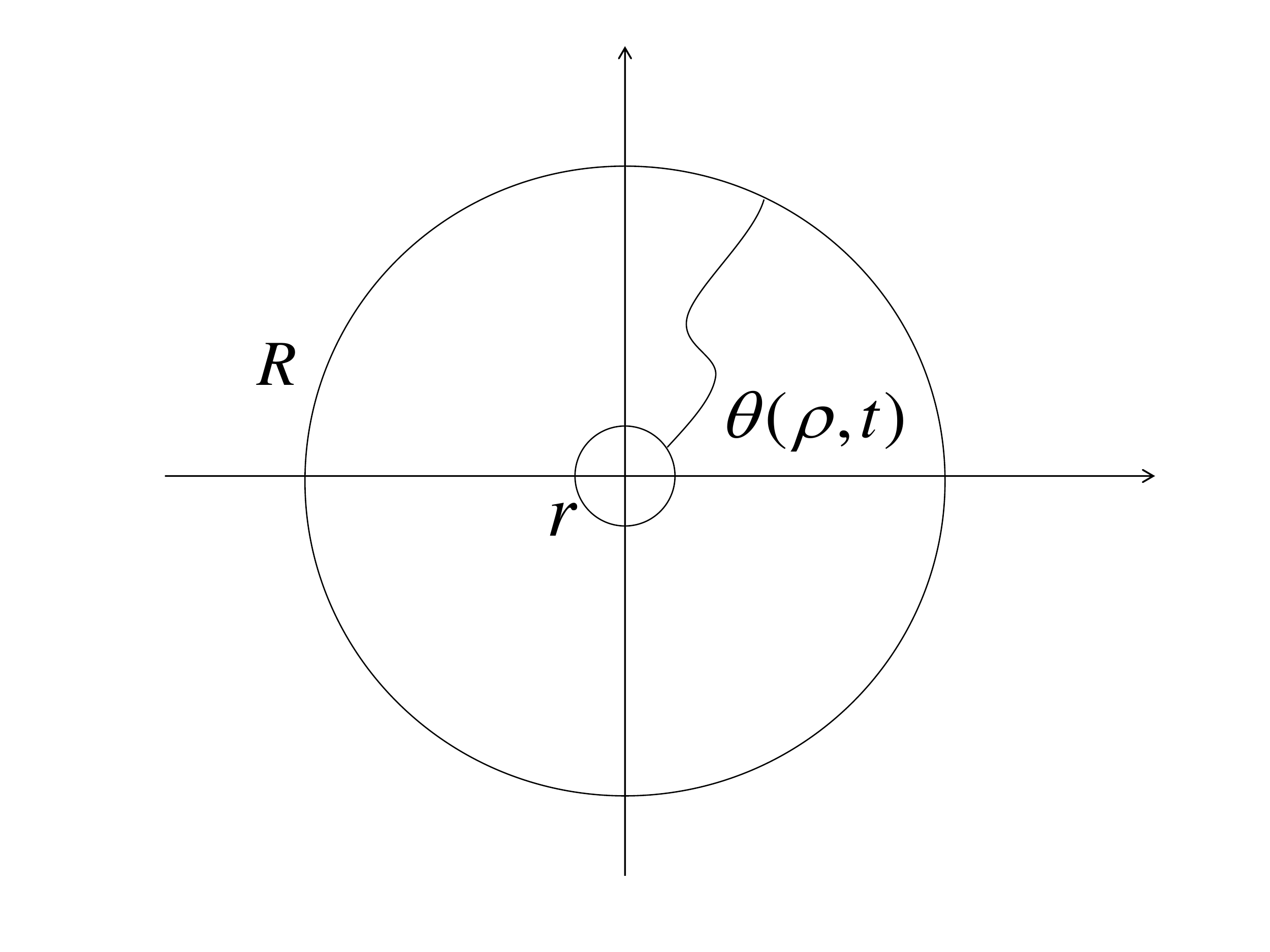}
		\vskip 0pt
		\caption{Research in \cite{GIK}}
        \label{fig:dighole}
	\end{center}
\end{figure}

{\bf Motivation of this research.} Ohtsuka, Goto and Nakagawa first prove the existence and uniqueness of spiral crystal growth for (\ref{eq:meancur}) by level set method in \cite{GNO} and \cite{O}. But they also consider this problem by digging a hole near the fixed points. Recently, \cite{OTG} simulates the level set of the solution given in \cite{O} by numerical method. In their paper, for $a>1/A$, the level set evolves as in Figure \ref{fig:introduction1}, \ref{fig:introduction2} and \ref{fig:introduction3}.

\begin{figure}[htbp]
	\begin{center}
            \includegraphics[height=5.0cm]{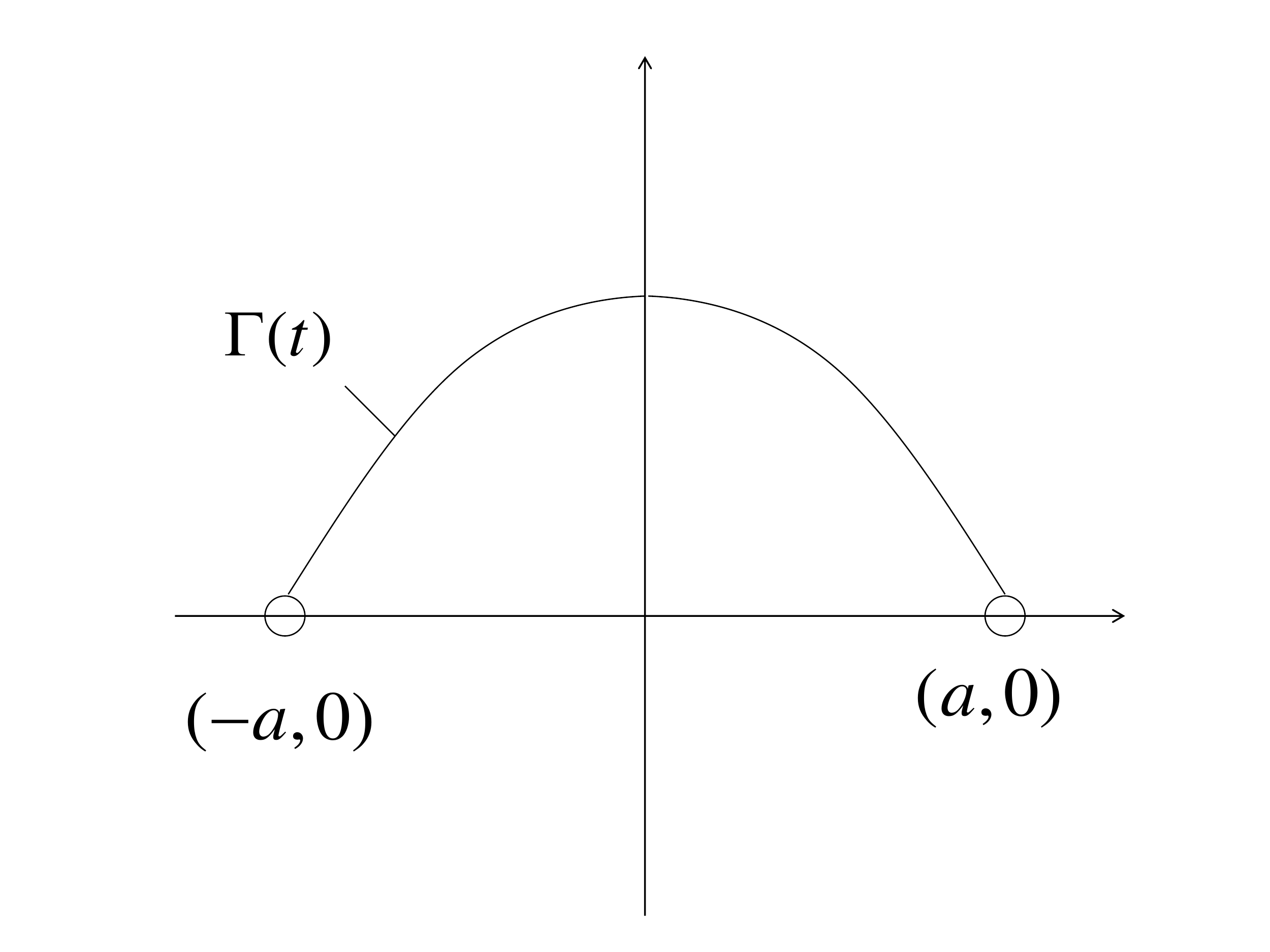}
		\vskip 0pt
		\caption{Evolution of level set $\#$1}
        \label{fig:introduction1}
	\end{center}
\end{figure}

\begin{figure}[htbp]
	\begin{center}
            \includegraphics[height=5.0cm]{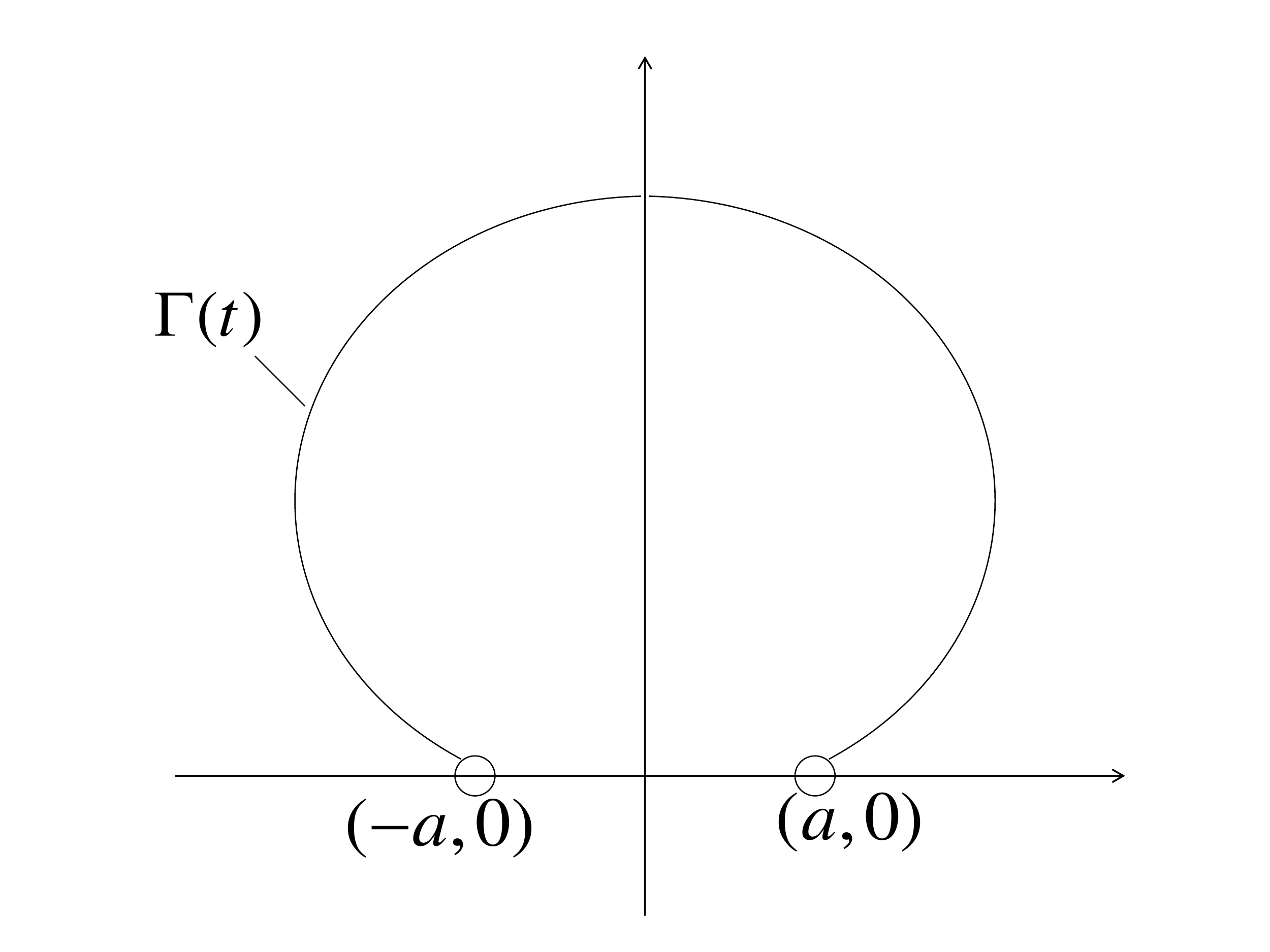}
		\vskip 0pt
		\caption{Evolution of level set $\#$2}
        \label{fig:introduction2}
	\end{center}
\end{figure}

\begin{figure}[htbp]
	\begin{center}
            \includegraphics[height=5.0cm]{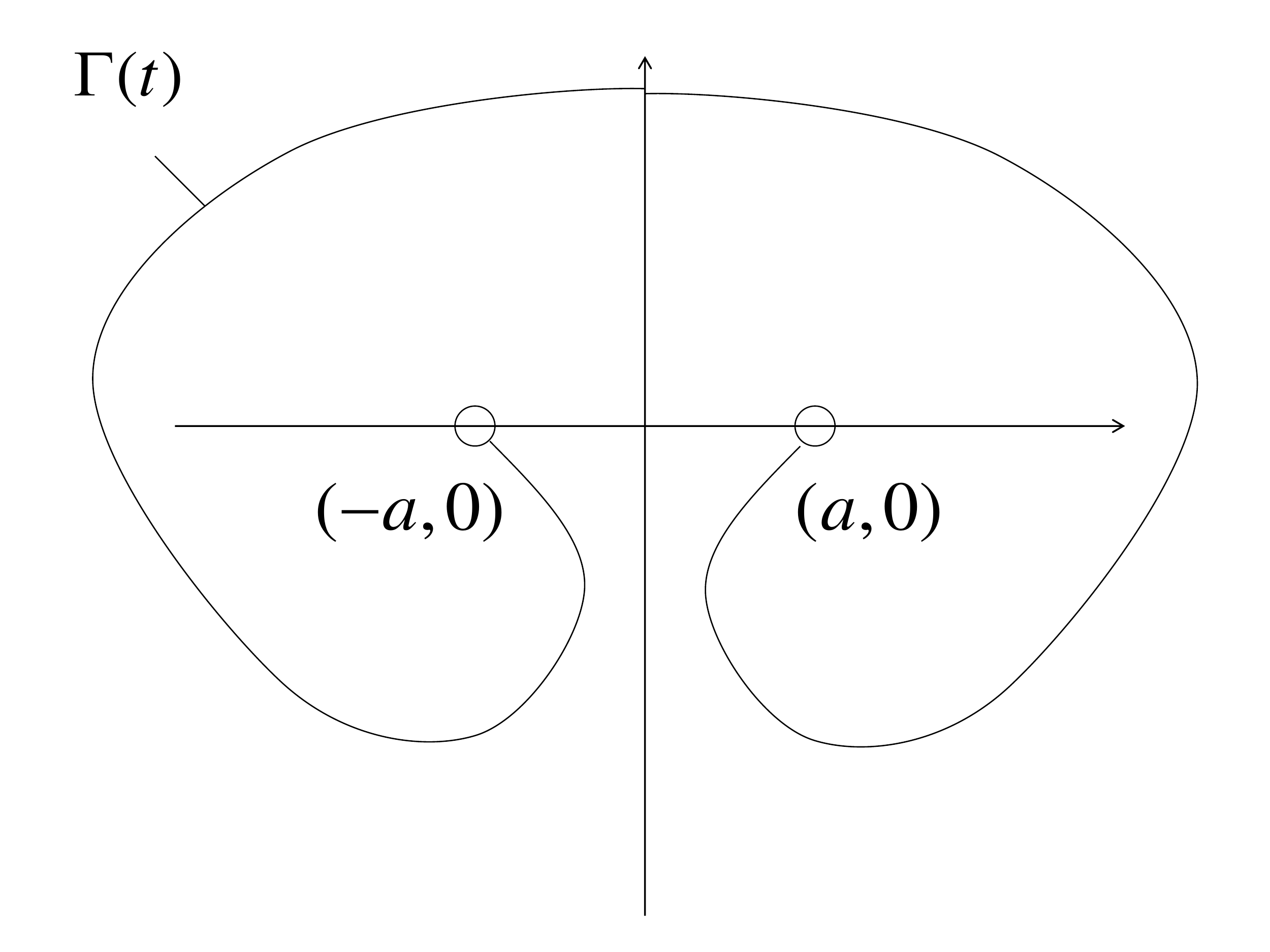}
		\vskip 0pt
		\caption{Evolution of level set $\#$3}
        \label{fig:introduction3}
	\end{center}
\end{figure}

Although, in our paper, we only consider the problem under the condition $a\leq 1/A$, the simulated results in \cite{OTG} give the hit about this research. We are devoted to considering them in analytic way.

The rest of this paper is organized as follows. In Section 2, we give some preliminary knowledge including the definition of semi-order, comparison principle and intersection number principle. In Section 3, we give the existence and uniqueness result for the fixed extreme points problem. Moreover, in Lemma \ref{lem:regular}, we give a sufficient condition for the solution $\Gamma_{\sigma}(t)$ remaining regular. In Section 4, we give the asymptotic behavior of the solution $\Gamma_{\sigma}(t)$ when $\sigma$ is large or small. Lemma \ref{lem:sigulartime} gives an important result for classifying $\Gamma_{\sigma}(t)$ by intersection number. In Section 5, we prove the asymptotic behavior for the condition (3) in Lemma \ref{lem:sigulartime} by Lyapunov function. In Section 6, we give the proof of Theorem \ref{thm:category}.

\section{Preliminary}
{\bf Semi-order} We want to define a semi-order for curves with the same fixed extreme points.

\begin{defn}\label{def:semiorder}
For any points $P$, $Q\in \mathbb{R}^2$ and $P\neq Q$, assume that maps $F_i(s)\in C([0,l_i]\rightarrow\mathbb{R}^2)$ and $F_i$ are differentiable at $0$ and $l_i$.The curves $\gamma_i$ are given by $\gamma_i=\{F_i(s)\mid 0\leq s\leq l_i, F_i(0)=P,F_i(l_i)=Q\}$, where $l_i$ is the length of $\gamma_i$, $i=1,2$. It is easy to see that $\gamma_i$ have the same extreme points $P$, $Q$, $i=1,2$. We say $\gamma_1\succ\gamma_2$, if 

(1). There exists connect, bounded and open domain $\Omega$ such that $\partial \Omega=\gamma_1\cup\gamma_2$;

(2). $\frac{d}{ds}F_1(0)\cdot\frac{d}{ds}F_2(0)\neq 1$ and $\frac{d}{ds}F_1(l_1)\cdot\frac{d}{ds}F_2(l_2)\neq1$;

(3). The domain $\Omega$ is located in the right hand side of $\gamma_1$, when someone walks along $\gamma_1$ from $P$ to $Q$. 

Where ``$\cdot$'' denotes the inner product in $\mathbb{R}^2$. We say $\gamma_1\succeq\gamma_2$, if there exist two sequences of curves $\{\gamma_{in}\}_{n\geq1}$, $i=1,2$ such that 

(1). $\lim\limits_{n\rightarrow\infty}d_H(\gamma_{in},\gamma_i)\rightarrow0$, $i=1,2$;

(2). $\gamma_{1n}\succ\gamma_{2n}$, $n\geq 1$.

Where $d_H(A,B)$ denotes the Hausdorff distance for set $A,B\subset \mathbb{R}^2$.
\end{defn}

\begin{figure}[htbp]
	\begin{center}
            \includegraphics[height=6.0cm]{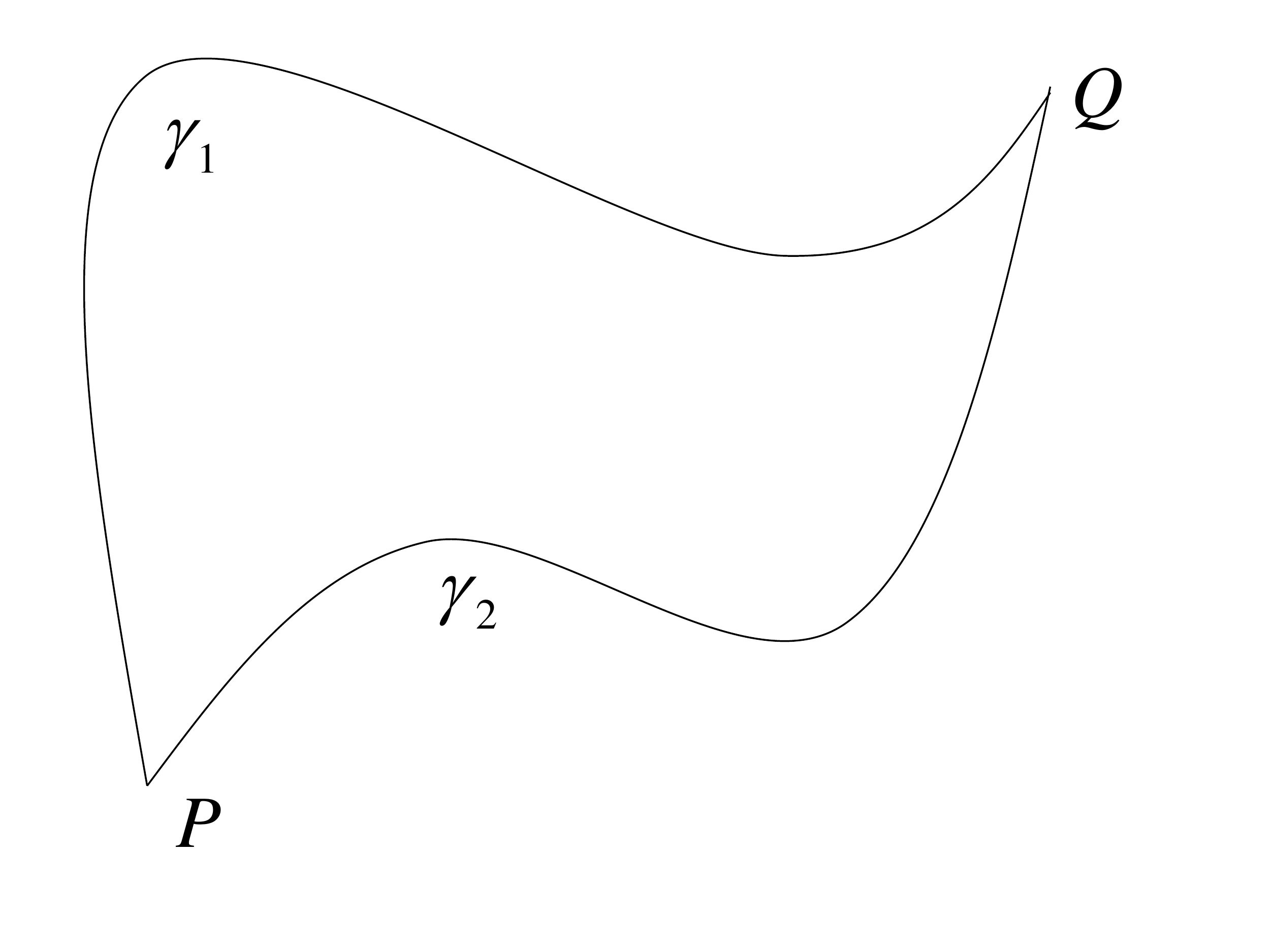}
		\vskip 0pt
		\caption{Definition \ref{def:semiorder}}
        \label{fig:semiorder}
	\end{center}
\end{figure}

Let $F(s)\in C^2([0,l]\rightarrow\mathbb{R}^2)$ and $\gamma=\{F(s)\mid s\in[0,l], F(0)=P, F(l)=Q\}$. Using the definition of semi-order, we can define a shuttle neighbourhood of $\gamma$. Seeing the assumption of $\gamma$, we can extend $\gamma$ by $\gamma^*$ such that $\gamma^*$ is $C^1$ curve and divides $\mathbb{R}^2$ into two connect parts denoted by $\Omega_1$ and $\Omega_2$. Moreover, $\Omega_1$ is located in the left hand side when someone walks along $\gamma^*$ from $P$ to $Q$.

\begin{rem}\label{rem:direction}
We say the normal vector of $\gamma$ is \textit{upward}(\textit{downward}), if the normal vector points to the domain $\Omega_1$($\Omega_2$).
\end{rem}

\begin{defn}[Shuttle neighbourhood]\label{def:shuttlenei}
We say $V$ is a \textit{shuttle neighbourhood} of $\gamma$, if there exist $\gamma_1$ and $\gamma_2$ such that

(1). $\gamma_i\subset \Omega_i$, $i=1,2$;

(2). $\gamma_1\succ\gamma\succ\gamma_2$;

(3). $\partial V=\gamma_1\cup \gamma_2$.
\end{defn}

{\bf Comparison principle and intersection number principle} Here we introduce the comparison principle and intersection number principle. The intersection number principle can help us classify the solutions.

For giving comparison principle, we must define sub,super-solution of (\ref{eq:meancurpara}).

\begin{defn}\label{def:subsup}
We say a continuous family of continuous curves $\{\gamma(t)\}$ is a sub(super)-solution of (\ref{eq:meancurpara}) and (\ref{eq:fixbound}), if

(1). $\gamma(t)$ are continuous curves and have the same extreme points $P$, $Q$;

(2). Let $\{S(t)\}$ be a smooth flow with extreme points $P$, $Q$. For some point $P^*$ and some time $t_0>0$ satisfying $P^*\in \gamma(t_0)$ but $P^*\neq P,\ Q$, if near the point $P^*$ and time $t_0$, $\{S(t)\}$ only intersects $\{\gamma(t)\}$ at $P^*$ and time $t_0$ from above(below). Let $V_{S(t)}$ denote the upward normal velocity of $S(t)$ and $\kappa_{S(t)}(P)$ denote the curvature at $P\in S(t)$. Then
$$
V_{S(t_0)}(P^*)\leq(\geq) -\kappa_{S(t_0)}(P^*)+A.
$$ 
\end{defn}

\begin{thm}[Comparison principle]\label{thm:comparison}
For two families of curves $\{\gamma_1(t)\}_{0\leq t\leq T}$ and $\{\gamma_2(t)\}_{0\leq t\leq T}$, assume $\{\gamma_1(t)\}_{0\leq t\leq T}$ is a super-solution of (\ref{eq:meancurpara}) and (\ref{eq:fixbound}), $\{\gamma_2(t)\}_{0\leq t\leq T}$ is a sub-solution of (\ref{eq:meancurpara}) and (\ref{eq:fixbound}). If $\gamma_1(0)\succeq\gamma_2(0)$, then $\gamma_1(t)\succ\gamma_2(t)$, $0\leq t\leq T$.
\end{thm}
We can prove this theorem by contradiction. Using local coordinate representation, by maximum principle and Hopf lemma, the conclusion can be got easily. Here we omit the detail.

In this paper, besides intersection number $Z[\cdot,\cdot]$, we introduce a related
notion $SGN[\cdot,\cdot]$(first used by \cite{DGM}), which turns out to be exceedingly useful in classifying the types of the solutions.
\begin{defn}\label{def:signal}
For two curves $\gamma_1$ and $\gamma_2$ satisfying the same conditions in Definition \ref{def:semiorder}, we define:

(1). $Z[\gamma_1,\gamma_2]$ is the number of the intersections between curves $\gamma_1$ and $\gamma_2$. Noting that $\gamma_1$ and $\gamma_2$ have the same extreme points, then $Z[\gamma_1,\gamma_2]\geq2$;

(2). $SGN[\gamma_1,\gamma_2]$ is defined when $Z[\gamma_1,\gamma_2]<\infty$. Denoting $n+1:=Z[\gamma_1,\gamma_2]<\infty$, let $P=P_0$, $P_1$, $\cdots$, $P_{n-1}$, $P_n=Q$ be the intersections. Here we assume $\wideparen{P_{i+1}P_0}>\wideparen{P_iP_0}$ and $\widetilde{P_{i+1}P_0}>\widetilde{P_iP_0}$, $i=1,\cdots,n$, where $\wideparen{P_iP_j}$ denotes the arc length of $\gamma_1$ between $P_i$ and $P_j$; $\widetilde{P_iP_j}$ denotes the arc length of $\gamma_2$ between $P_i$ and $P_j$. If $\gamma_1\mid_{\wideparen{P_iP_{i-1}}}\succ\gamma_2\mid_{\widetilde{P_iP_{i-1}}}$, we say the sign between $P_i$ and $P_{i-1}$ is ``$+$''; Respectively, $\gamma_2\mid_{\widetilde{P_iP_{i-1}}}\succ\gamma_1\mid_{\wideparen{P_iP_{i-1}}}$, we say the sign between $P_i$ and $P_{i-1}$ is ``$-$'', $i=1,\cdots,n$. Where $\gamma_1\mid_{\wideparen{P_iP_{i-1}}}$ and $\gamma_2\mid_{\widetilde{P_iP_{i-1}}}$ denote the restriction between $P_{i-1}$ and $P_{i}$.

 $SGN[\gamma_1,\gamma_2]$ called \textit{ordered word set} consists the sign between $P_i$ and $P_{i-1}$, $i=1,\cdots,n$. 
\end{defn}

For explaining Definition \ref{def:signal}, we give an example. Seeing Figure \ref{fig:sign}, $Z[\gamma_1,\gamma_2]=6$ and 
$$
SGN[\gamma_1,\gamma_2]=[-\ +\ -\ +\ -].
$$
\begin{figure}[htbp]
	\begin{center}
            \includegraphics[height=6.0cm]{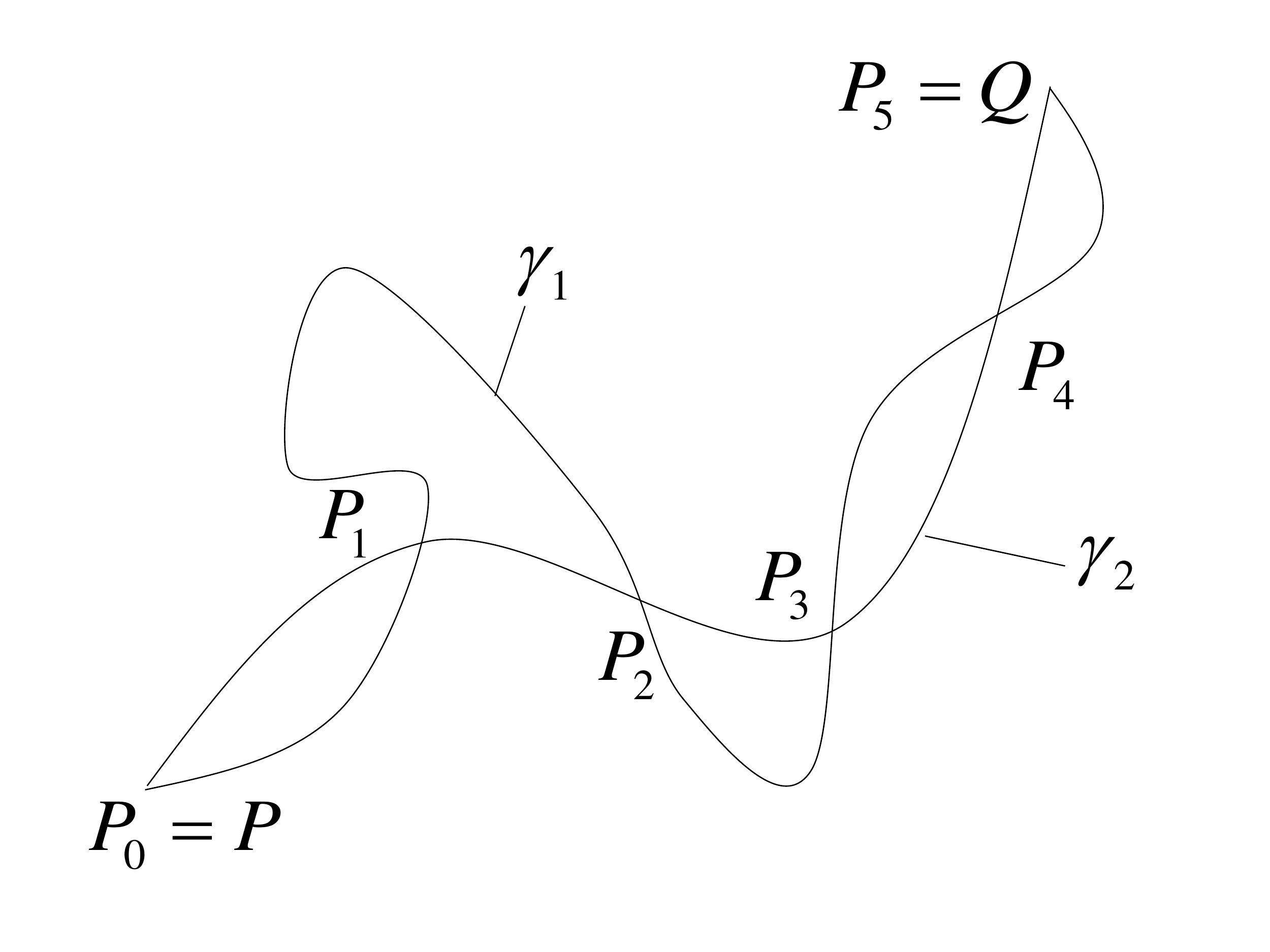}
		\vskip 0pt
		\caption{Example for $SNG[\cdot,\cdot]$}
        \label{fig:sign}
	\end{center}
\end{figure}

Let $A$ and $B$ be two ordered word sets, we write $A\triangleright B$, if $B$ is a sub ordered word set of $A$. For example,
$$
[+\ -]\triangleright B\ \text{for}\ B=[+\ -],\ [+],\ [-],\ \text{but\ not}\ [+\ -]\triangleright[-\ +].
$$
\begin{rem}\label{rem:nothold}
For the curve shortening flow with driving force, even if $\gamma_1(t)$ and $\gamma_2(t)$ satisfy (\ref{eq:meancurpara}) and (\ref{eq:fixbound}), we can not guarantee that for all $t_1<t_2$,
$$
Z[\gamma_1(t_2),\gamma_2(t_2)]\leq Z[\gamma_1(t_1),\gamma_2(t_1)],\ SGN[\gamma(t_2),\gamma(t_2)]\triangleleft SGN[\gamma(t_1),\gamma(t_1)].
$$
\end{rem}

For giving the intersection number principle, we need assume $\gamma_1(t)$ and $\gamma_2(t)$ are homeomorphism to a curve.

\begin{thm}[Intersection number principle]\label{thm:intersection}
For two families of curves $\{\gamma_1(t)\}_{0\leq t\leq T}$ and $\{\gamma_2(t)\}_{0\leq t\leq T}$ satisfying (\ref{eq:meancurpara}) and (\ref{eq:fixbound}), assume there exist a $C^1$ curve $M$ with extreme points $P$, $Q$ and two maps 
$$
\varphi_1,\ \varphi_2:M\times[0,T]\rightarrow\mathbb{R}^2
$$ 
such that 
$$
\gamma_i(t)=\{\varphi_i(P,t)\mid P\in M \},\ i=1,2. 
$$
Then there hold 
$$
Z[\gamma_1(t_2),\gamma_2(t_2)]\leq Z[\gamma_1(t_1),\gamma_2(t_1)],\ SGN[\gamma(t_2),\gamma(t_2)]\triangleleft SGN[\gamma(t_1),\gamma(t_1)],
$$
for all $0\leq t_1<t_2\leq T$.
\end{thm}
Using the arc length parameter $s$ of $M$, we can express $\gamma_i(t)$ by 
$$
\gamma_i(t)=\{\varphi_i(s,t)\mid 0\leq s\leq L \},\ 0\leq t\leq T,\ i=1,2,
$$ 
where $L$ denotes the length of $M$. Using the local representation and classical intersection number principle, we can prove this results easily. We omit the detail.

\begin{defn}\label{def:C1close}
For a $C^1$ curve $\gamma$ and a sequence of $C^1$ curves $\gamma_n$ with extreme points $P$, $Q$, we say $\gamma_n\rightarrow\gamma$ in $C^1$, if 

(1) There exist a $C^1$ curve $M$ with extreme points $P$, $Q$ and maps 
$$
\varphi,\ \varphi_n:M\rightarrow\mathbb{R}^2
$$ 
such that 
$$
\gamma=\{\varphi(P)\mid P\in M \},\ \gamma_n=\{\varphi_n(P)\mid P\in M \}. 
$$

(2) 
$$
\left\Vert \varphi_n-\varphi \right\Vert_{C^1(M\rightarrow \mathbb{R}^2)}\rightarrow 0,
$$
as $n\rightarrow\infty$.
\end{defn}

\section{Time local existence and uniqueness of solution}\large

In this section, we introduce the the transport map first used by \cite{A} and prove Theorem \ref{thm:exist}.

\begin{lem}\label{lem:vectorfield}
For $\Gamma_0$ satisfying the assumption in Theorem \ref{thm:exist}, there exist a shuttle neighbourhood $V$ of $\Gamma_0$ and a vector field $X\in C^1(\overline{V}\rightarrow\mathbb{R}^2)$ such that 
$$
X(z)\cdot n(z) >0,\ z\in \Gamma_0
$$
and in $V$, there holds 
$$
|X|\geq\delta>0, \text{for\ some}\ \delta>0, 
$$
where $n$ denotes the unit upward normal vector of $\Gamma_0$.
\end{lem}
\begin{proof}
We extend $\Gamma_0$ by $\Gamma_0^*$ such that $\Gamma_0^*$ is a $C^2$ curve and divide $\mathbb{R}^2$ into two connect parts $\Omega_1$ and $\Omega_2$. Assume $\Omega_1$($\Omega_2$) locates in the left(right) side of $\Gamma_0^*$(``left side'' and ``right side'' are defined as in Section 2). 

Let $d(x)$ be the signed distance function defined as following:
$$
d(x)=d(x,\Omega_2)-d(x,\Omega_1),\ x\in\mathbb{R}^2.
$$ 
Since $\Gamma_0^*$ is $C^2$, as we know, there exists a tubular neighbourhood $U$ of $\Gamma_0^*$ such that $d$ is $C^2$ in $U$. Moreover, there exists a projection map $P$ such that for all $z\in U$ there exists a unique point $z^*\in \Gamma_0^*$ such that 
$$
P z=z^*
$$
and $\nabla d(z)=\nabla d(z^*)=n(z^*)$. We choose two curves $\Gamma_1,\Gamma_2\subset U$ and $\Gamma_i\subset \Omega_i$, $i=1,2$, such that $\Gamma_1\succ\Gamma_0\succ\Gamma_2$. Let $V$ be the domain satisfying $\partial V=\Gamma_1\cup\Gamma_2$ and $X(z)=\nabla d(z)$. Obviously 
$$
|X|(z)=1,\ z\in V
$$
and 
$$
X(z)\cdot n(z)=1,\ z\in \Gamma_0.
$$
\end{proof}

{\bf Transport map} Let $\phi:\Gamma_0\times(-\delta,\delta)\rightarrow V$ be the map generated by vector field $X$, precisely,
$$
\left\{\begin{array}{lcl}
\frac{d}{d\alpha}\phi(P,\alpha)=X(\phi),\ P\in \Gamma_0, \\
\sigma(P,0)=P,\ P\in\Gamma_0.
\end{array}
\right.
$$
Recalling $\Gamma_0=\{F_0(s)\mid 0\leq s\leq L_0\}$ and $F_0(s)\in C^2([0,L_0]\rightarrow\mathbb{R}^2)$, let 
$$
\psi(s,\alpha)=\phi(F_0(s),\alpha).
$$ 
Seeing the assumption of $F_0$ and $X$, $\psi_s$, $\psi_{\alpha}$, $\psi_{ss}$, $\psi_{s\alpha}$, $\psi_{\alpha\alpha}$ are all continuous vectors for $0\leq s\leq L_0$, $-\delta<\alpha<\delta$.

If $\Gamma(t)\subset V$ is $C^1$ close to $\Gamma_0$ and satisfies (\ref{eq:meancurpara}), (\ref{eq:fixbound}), $0<t<T$ with initial data $\Gamma(0)=\Gamma_0$, then there exists a function $u(\cdot,t):[0,L_0]\rightarrow\mathbb{R}$ such that 
$$
\Gamma(t)=\{\psi(s,u(s,t))\mid 0\leq s\leq L_0\}.
$$
Moreover, $u$ satisfies
\begin{equation}\label{eq:localpara}
\left\{
\begin{array}{lcl}
\dis{u_t=\frac{1}{|\psi_s+\psi_{\alpha}u_s|^2}u_{ss}+\frac{\det(\psi_s+\psi_{\alpha}u_s,\psi_{ss}+2u_s\psi_{s\alpha}+\psi_{\alpha\alpha}u_s^2)}{\det(\psi_s,\psi_{\alpha})|\psi_s+\psi_{\alpha}u_s|^2}+A\frac{|\psi_{s\alpha+\psi_{\alpha}u_s}|}{\det(\psi_s,\psi_{\alpha})}},\\
0<s<L_0,\ 0<t<T,\\
u(0,t)=u(L_0,t)=0,\ 0\leq t<T\\
u(s,0)=0,\ 0\leq s\leq L_0
\end{array}
\right.
\end{equation}
where $\det(\cdot,\cdot)$ denotes the determinant. Indeed, the upward normal velocity 
$$
V=\frac{\det(\psi_s,\psi_{\alpha})u_t}{|\psi_s+\psi_{\alpha}u_s|}
$$
and the curvature 
$$
\kappa=\frac{\det(\psi_s,\psi_{\alpha})}{|\psi_s+\psi_{\alpha}u_s|^3}u_{ss}+\frac{\det(\psi_s+\psi_{\alpha}u_s,\psi_{ss}+2u_s\psi_{s\alpha}+\psi_{\alpha\alpha}u_s^2)}{|\psi_s+\psi_{\alpha}u_s|^3}.
$$

Following Proposition \ref{pro:localexist} implies Theorem \ref{thm:exist}.

\begin{prop}\label{pro:localexist}
There exist $T_0>0$ and a unique $u\in C([0,L_0]\times[0,T_0))\cap C^{2+\alpha,1+\alpha/2}([0,L_0]\times(0,T_0))$ such that $u$ satisfies (\ref{eq:localpara}) for $T=T_0$.
\end{prop}

\begin{proof}
Since $\psi_{s}(s,0)\cdot\psi_{\alpha}(s,0)=0$, $0\leq s\leq L_0$, then
$$
|\det(\psi_s,\psi_{\alpha})|(s,0)=1,\ 0\leq s\leq L_0.
$$
There exist $\delta_1>0$ and $\alpha_0$ such that for all $-\alpha_0<\alpha<\alpha_0$ and $0\leq s\leq L_0$,
$$
|\det(\psi_s,\psi_{\alpha})|(s,\alpha)>\delta_1.
$$
By the quasi-linear parabolic theory in \cite{LSU}, we can deduce there exist $T_0$ and $u\in C([0,L_0]\times[0,T_0))\cap C^{2+\alpha,1+\alpha/2}([0,L_0]\times(0,T_0))$ such that $u$ satisfies (\ref{eq:localpara}) and $|u|\leq \alpha_0$, $0\leq t<T_0$.
For the uniqueness, since $\psi_s$, $\psi_{\alpha}$, $\psi_{ss}$, $\psi_{s\alpha}$, $\psi_{\alpha\alpha}$ are all continuous vectors for $0\leq s\leq L_0$, $-\alpha_0<\alpha<\alpha_0$, the unique result can be got easily.
\end{proof}

\begin{rem}\label{rem:weak}
The assumption for initial curve can be weakened. In this paper, we assume $F_0(s)\in C^2([0,L_0]\rightarrow\mathbb{R}^2)$. Indeed, the initial curve can be assumed to be Lipschitz continuous. Recently, \cite{M} has considered the curve-shortening flow with Lipschitz initial curve, under the Neumann boundary condition. Since the purpose of this paper is to get the three categories of solutions, we do not introduce this part in detail.
\end{rem}

\begin{lem}\label{lem:curvaturee}
For $\Gamma(t)$ satisfying (\ref{eq:meancurpara}), (\ref{eq:fixbound}), for $0<t<T$, then the curvature $\kappa(s,t)$ satisfies
\begin{equation}\label{eq:curvaturee}
\left\{
\begin{array}{lcl}
\kappa_t=\kappa_{ss}-\kappa\kappa_s^2+\kappa^2(\kappa-A),\ 0<s<L(t),\ 0<t<T\\
\kappa(0,t)=A,\ \kappa(L(t),t)=A, 0<t<T,
\end{array}
\right.
\end{equation}
where $\kappa_t$ denotes the derivative of $t$ by fixing $s$.
\end{lem}
For the proof of the first equation, the calculation can be seen in \cite{GMSW}. Since at the extreme points, $\Gamma(t)$ does not move, the boundary condition is obvious.

\begin{lem}\label{lem:regular}
For $\sigma>0$, $\Gamma_{\sigma}(t)$ given in Theorem \ref{thm:category}, let $F_{\sigma}(s,t)$ satisfy
$$
\Gamma_{\sigma}(t)=\{F_{\sigma}(s,t)\mid0\leq s\leq L_{\sigma}(t)\},
$$
where $L_{\sigma}(t)$ is the length of $\Gamma_{\sigma}(t)$. If $\frac{d}{ds} F_{\sigma}(0,t)\cdot(0,1)>0$, for all $0\leq t\leq t_0$, then $t_0<T_{\sigma}$.
\end{lem}
This lemma gives a sufficient condition under which $\Gamma_{\sigma}(t)$ does not become singular. The assumption $\frac{d}{ds} F_{\sigma}(0,t)\cdot(0,1)>0$ means that the $y$-component of the tangential vector $\frac{d}{ds} F_{\sigma}(0,t)$ is positive.

\begin{proof}
Seeing the choice of $\Gamma_{\sigma}(0)$, then $\kappa_{\sigma}(s,0)\geq0$, $0\leq s\leq L_{\sigma}(0)$, for $\sigma>0$. Combining Lemma \ref{lem:curvaturee} and maximum principle, $\kappa_{\sigma}(s,t)>0$, $0<s<L(t)$, $0<t<T_{\sigma}$.

If $T_{\sigma}=\infty$, the result is trivial. We assume $T_{\sigma}<\infty$. 

We prove the result by contradiction, assuming $t_0\geq T_{\sigma}$. We claim that every half-line given by
$$
y=kx,\ y\geq0,\ \text{or}\ x=0,\ y\geq 0
$$
intersects $\Gamma_{\sigma}(t)$ only once, $0<t<T_{\sigma}$.

First, for all $0<t<T_{\sigma}$, we prove $x=0$, $y\geq0$ intersects $\Gamma_{\sigma}(t)$ only once. If not, suppose that there exists $t_1<T_{\sigma}$ such that $x=0$, $y\geq0$ intersects $\Gamma_{\sigma}(t_1)$ more than once. Since $\Gamma_{\sigma}(t)$ is symmetric about $y$-axis, it is easy to see that $\Gamma_{\sigma}(t)$ becomes singular at $t_1$. This contradicts to $t_1<T_{\sigma}$.

Next, by contradiction, assume that there exist $t_2<T_{\sigma}$ and $k<0$ such that $y=kx,\ y\geq0$ intersects $\Gamma_{\sigma}(t_2)$ more than once. Combining our assumption $\frac{d}{ds} F_{\sigma}(0,t)\cdot(0,1)>0$, we can choose $k_0$ satisfying $k_0<k<0$ such that half-line $y=k_0x,\ y\geq0$ intersects $\Gamma_{\sigma}(t_2)$ tangentially at some point $P^*$ and near $P^*$, $\Gamma_{\sigma}(t_2)$ is located under the half-line. It is easy to deduce that the curvature at $P^*$, $\kappa_{\sigma}(P^*,t_2)\leq0$. This contradicts to that the curvatures on $\Gamma_{\sigma}(t)$ are all positive, $0<t<T_{\sigma}$. Here we complete the proof of claim.
\begin{figure}[htbp]
	\begin{center}
            \includegraphics[height=6.0cm]{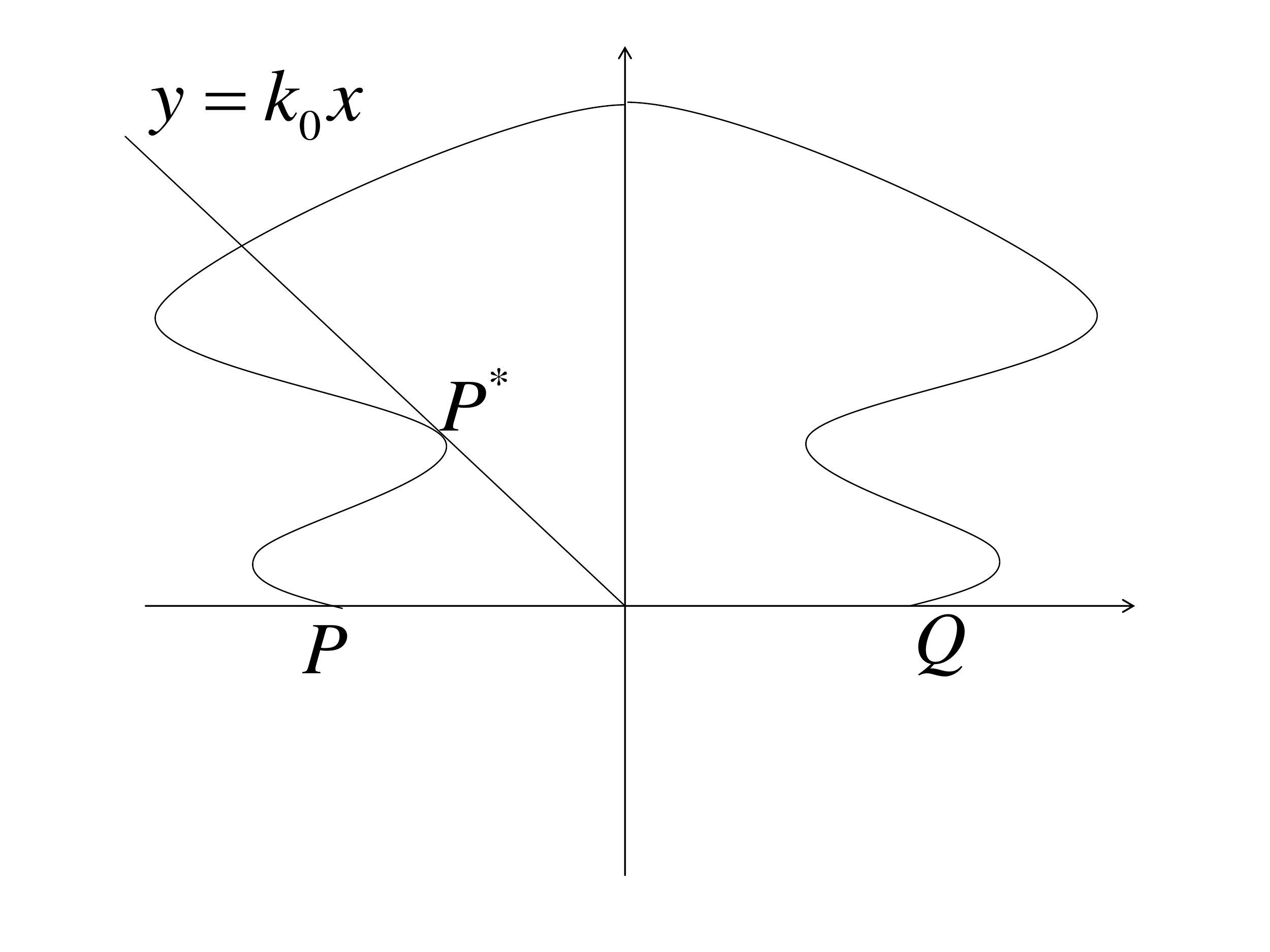}
		\vskip 0pt
		\caption{Proof of claim}
        \label{fig:polarcoordinate}
	\end{center}
\end{figure}

Seeing $\frac{d}{ds} F_{\sigma}(0,t)\cdot(0,1)>0$ and the claim above, $\Gamma_{\sigma}(t)\subset \{(x,y)\mid y\geq 0\}$, $t<T_{\sigma}$. The claim implies that we can express $\Gamma_{\sigma}(t)$ by polar coordinate. For $(x,y)\in \Gamma_{\sigma}(t)$, let
$$
\left\{
\begin{array}{lcl}
x=\rho_{\sigma}(\theta,t)\cos\theta,\\
y=\rho_{\sigma}(\theta,t)\sin\theta,
\end{array}
\right.
$$
for $0\leq\theta\leq \pi$, $0\leq t<T_{\sigma}$. Consequently, $\rho_{\sigma}$ satisfies
\begin{equation}\label{eq:polarco}
\left\{
\begin{array}{lcl}
\dis{\rho_{t}=\frac{\rho_{\theta\theta}}{\rho^2+\rho^2_{\theta}}-\frac{2\rho^2_{\theta}+\rho^2}{\rho(\rho^2_{\theta}+\rho^2)}+\frac{1}{\rho}A\sqrt{\rho^2_{\theta}+\rho^2}},\ 0<\theta<\pi,\ 0<t<T_{\sigma},\\
\rho(0,t)=a,\ \rho(\pi,t)=a,\ 0\leq t<T_{\sigma},
\end{array}
\right.
\end{equation}
recalling $P=(-a,0)$, $Q=(a,0)$. 

Since for $\sigma>0$, $\Gamma_{\sigma}(0)\succ \Lambda_0=\{(x,y)\mid y=0,\ -a\leq x\leq a\}$. It is easy to see that $\Lambda_0$ is a sub-solution of (\ref{eq:meancurpara}) and (\ref{eq:fixbound}). By comparison principle, $\Gamma_{\sigma}(t)\succ\Lambda_0$ for $t<T_{\sigma}$. This implies that there exist $t_3>0$ and $\rho_1>0$ such that $\rho_{\sigma}(\theta,t)\geq \rho_1$ for $t_3<t<T_{\sigma}$, $0\leq \theta\leq \pi$. On the other hand, since $T_{\sigma}<\infty$, there exists $\rho_2>0$ such that $\rho_{\sigma}(\theta,t)\leq \rho_2$ for $0<t<T_{\sigma}$, $0\leq \theta\leq \pi$. Therefore, the quasilinear theory in \cite{LSU} shows that for $\epsilon>0$, there exists $C_{\epsilon}$ such that 
$$
\left\Vert \rho_{\sigma}(\cdot,t)\right\Vert_{C^2[0,\pi]}\leq C_{\epsilon},\ t_3+\epsilon<t<T_{\sigma}.
$$ 
Therefore the curvature of $\Gamma_{\sigma}(t)$ is uniform bounded for $t$ close to $T_{\sigma}$. This implies that the solution $\Gamma_{\sigma}(t)$ can be extended over time $T_{\sigma}$. This contradicts to that $T_{\sigma}$ is the maximal existence time.
\end{proof}

\begin{lem}[Continuous dependence on the initial curve]\label{lem:localconti}
Assume $\rho$ and $\rho_n$ are the solution of (\ref{eq:polarco}) for $0\leq \theta\leq \pi$, $0<t<T$. If $\rho$ is bounded from below for some positive constant and
$$
\lim\limits_{n\rightarrow\infty}\left\Vert\rho_{n}(\cdot,0)-\rho(\cdot,0) \right\Vert_{C^1[0,\pi]}=0,
$$
then for all $0<t<T$,
$$
\lim\limits_{n\rightarrow\infty}\left\Vert\rho_{n}(\cdot,t)-\rho(\cdot,t) \right\Vert_{C^2[0,\pi]}=0.
$$
\end{lem}

\section{Behavior for $\sigma$ sufficient small or large}
\begin{prop}\label{prop:sigmalarge}
There exists $\sigma_1>0$ such that for all $\sigma>\sigma_1$, there exists some time $T_{\sigma}^*<T_{\sigma}$ such that $\Gamma_{\sigma}(t)\succ \Gamma^*$, for $T_{\sigma}^*<t<T_{\sigma}$
\end{prop}

For proving this proposition, we introduce the Grim reaper for the curve shortening flow. Grim reaper is given by 
$$
G(x,t)=C-\frac{t}{b}+b\ln \cos\frac{x}{b},\ -\frac{b\pi}{2}<x<\frac{b\pi}{2},\ t>0,
$$
where $b>0$ and $C\in\mathbb{R}$. It is easy to see $G(x,t)$ satisfies
$$
G_t=\frac{G_{xx}}{1+G_x^2}.
$$
The Grim reaper $G(x,t)$ is a traveling wave moving downward with speed $1/b$.

\begin{lem}\label{lem:grimreaper}
If $b<2a/\pi$, the curve
$$
\gamma_G(t)=\{(x,y)\mid y=\max\{G(x,t),0\},\ |x|<\frac{b\pi}{2}\} \cup \{(x,y)\mid y=0,\ \frac{b\pi}{2}\leq|x|\leq a\}
$$
is a sub-solution of (\ref{eq:meancurpara}) and (\ref{eq:fixbound}) in the sense of Definition \ref{def:subsup}.
\end{lem}
\begin{proof}
When $0<t<bC$, let $x(t)>0$ such that $G(x(t),t)=0$. 

For $|x|<x(t)$, $\gamma_G=\{(x,y)\mid y=G(x,t)\}$. Therefore, 
$$
G_{t}\leq \frac{G_{xx}}{1+G_x^2}+A\sqrt{1+G_x^2},\ |x|<x(t).
$$
For $x(t)<|x|<a$, $\gamma_G=\{(x,y)\mid y=0\}$. Obviously, $y=0$ is a sub-solution of
$$
u_{t}\leq \frac{u_{xx}}{1+u_x^2}+A\sqrt{1+u_x^2},\ x(t)<|x|<a.
$$
At the point $x=x(t)$($x=-x(t)$), it is impossible that for smooth flow $S(t)$, near $x=x(t)$($x=-x(t)$), $S(t)$ touches $\gamma_{G}(t)$ at $x=x(t)$($x=-x(t)$) only once from above.

Therefore, $\gamma_G(t)$ is a sub-solution of (\ref{eq:meancurpara}) and (\ref{eq:fixbound}), for $0<t<bC$.

When $t\geq bC$, $\gamma_G=\Lambda_0=\{(x,y)\mid y=0,\ |x|\leq a\}$(given in the proof of Lemma \ref{lem:regular}). Obviously, $\gamma_G$ is a sub-solution of (\ref{eq:meancurpara}) and (\ref{eq:fixbound}), for $t\geq bC$.
\end{proof}

Following lemma gives the result for the classification of the solution $\Gamma_{\sigma}(t)$.

\begin{lem}\label{lem:sigulartime}
For $\Gamma_{\sigma}(t)$ given by Theorem \ref{thm:category}, for $\sigma>0$, $\Gamma_{\sigma}(t)$ satisfies one of the following four conditions:

(1). $SGN(\Gamma_{\sigma}(t),\Gamma^*)=[-]$ for $t<T_{\sigma}$. Moreover, $T_{\sigma}=\infty$;

(2). there exists $t^*_{\sigma}$ such that $SGN(\Gamma_{\sigma}(t),\Gamma^*)=[-\ +\ -]$ for $t<t^*_{\sigma}$ and $SGN(\Gamma_{\sigma}(t),\Gamma^*)=[-]$ for $t^*_{\sigma}<t<T_{\sigma}$. Moreover, $T_{\sigma}=\infty$;

(3). $SGN(\Gamma_{\sigma}(t),\Gamma^*)=[-\ +\ -]$ for $t<T_{\sigma}$. Moreover, $T_{\sigma}=\infty$;

(4). there exists $T_{\sigma}^*<T_{\sigma}$ such that $SGN(\Gamma_{\sigma}(t),\Gamma^*)=[-\ +\ -]$ for $t<T_{\sigma}^*$ and $SGN(\Gamma_{\sigma}(t),\Gamma^*)=[+]$, $T_{\sigma}^*<t<T_{\sigma}$.
\end{lem}

\begin{proof}
Seeing the assumption in Theorem \ref{thm:category}, there exists $\sigma_0>0$ such that

(a). $0<\sigma\leq\sigma_0$, $\Gamma^*\succeq\Gamma_{\sigma}$;

(b). $\sigma>\sigma_0$, $\Gamma_{\sigma}$ intersects $\Gamma^*$ fourth.

{\bf Step 1.} For $0<\sigma\leq\sigma_0$, by comparison principle, $\Gamma^*\succ\Gamma_{\sigma}(t)$, for $0<t<T_{\sigma}$. Noting $\sigma>0$, $\Gamma_{\sigma}(t)\succ \Lambda_0$, for $0<t<T_{\sigma}$. Therefore, $\frac{d}{ds}F_{\sigma}(0,t)\cdot(0,1)>0$, $0<t<T_{\sigma}$. By the same method in the proof of Lemma \ref{lem:regular}, we can prove $T_{\sigma}=\infty$. Therefore, for $0<\sigma\leq\sigma_0$, condition (1) holds. 

{\bf Step 2.} For $\sigma>\sigma_0$, seeing the choice of $\Gamma_{\sigma}$, $SGN(\Gamma_{\sigma},\Gamma^*)=[-\ +\ -]$.  

Let $\tau_0$ depending on $\sigma$ satisfy 
$$
\tau_0=\sup\{\tau \mid\frac{d}{ds}F_{\sigma}(0,t)\cdot(0,1)>0, 0<t<\tau\}. 
$$ 
Since $\frac{d}{ds}F_{\sigma}(0,0)\cdot(0,1)>0$, we can deduce $\tau_0>0$. Therefore, Lemma \ref{lem:regular} implies $T_{\sigma}>\tau_0$. Moreover, $\Gamma_{\sigma}(t)$ can be represented by polar coordinate, $0<t<\tau_0$. This means that $\Gamma_{\sigma}(t)$ satisfies the assumption of Theorem \ref{thm:intersection} for $0<t<\tau_0$. Then 
$$
SGN(\Gamma_{\sigma}(t),\Gamma^*)\triangleleft[-\ +\ -],\ 0<t<\tau_0.
$$ 
Seeing the symmetry of $\Gamma_{\sigma}(t)$, then for $t<\tau_0$, one of the following three conditions holds

(i). $SGN(\Gamma_{\sigma}(t),\Gamma^*)=[+]$;

(ii). $SGN(\Gamma_{\sigma}(t),\Gamma^*)=[-]$; 

(iii). $SGN(\Gamma_{\sigma}(t),\Gamma^*)=[-\ -]$.

(iv). $SGN(\Gamma_{\sigma}(t),\Gamma^*)=[-\ +\ -]$.

{\bf Step 3.} If for some $t^*_{\sigma}<\tau_0$ (ii) or (iii) holds, then $\Gamma^*\succeq\Gamma_{\sigma}(t^*_{\sigma})$. Then by comparison principle, $\Gamma^*\succ\Gamma(t)\succ\Lambda_0$, $t^*_{\sigma}<t<T_{\sigma}$. Therefore, by the same argument in Step 1, condition (2) holds. 

If for some $T_{\sigma}^*<\tau_0$ (i) holds, this means that $\Gamma_{\sigma}(T_{\sigma}^*)\succ\Gamma^*$. By comparison principle, $\Gamma_{\sigma}(t)\succ\Gamma^*$, $T_{\sigma}^*<t<T_{\sigma}$. Therefore, condition (4) holds.

If for every $t<\tau_0$, there holds $SGN(\Gamma_{\sigma}(t),\Gamma^*)=[-\ +\ -]$. Combining $\Gamma_{\sigma}(t)\succ\Lambda_0$, $t<\tau_0$, there exists $\delta>0$ such that
$$
\frac{d}{ds}F_{\sigma}(0,t)\cdot(0,1)>\delta,\ t<\tau_0.
$$
If $\tau_0<\infty$, by the definition of $\tau_0$, $\frac{d}{ds}F_{\sigma}(0,\tau_0)\cdot(0,1)=0$. This yields a contradiction. Therefore, $\tau_0=\infty$. Consequently, $T_{\sigma}=\infty$. Condition (3) holds.

We complete the proof.
\end{proof}

In the following, we consider two circles. 
$$
\partial B_1=\big\{(x,y)\mid (x-R)^2+\big(y-(1+R/a)\sqrt{1/A^2-a^2}\big)^2=R^2\big\}
$$
and 
$$
\partial B_2=\big\{(x,y)\mid \big(x-R)^2+(y-(1+R/a)\sqrt{1/A^2-a^2}\big)^2=\frac{1}{A^2}(1+R/a)^2\big\},
$$
where $R>1/A$. And we let
$$
(0,K)=\partial B_2\cap\{x=0\mid y>0\}.
$$
It is easy to check that $\partial B_2$ intersects $\Gamma^*$ tangentially at $(-a,0)$. Let $R(t)$ be the solution of 
\begin{equation}\label{eq:circle}
R^{\prime}(t)=A-1/R(t)
\end{equation}
with $R(0)=R$. Since $R(0)>1/A$, $R(t)$ is increasing and $\lim\limits_{t\rightarrow\infty}R(t)=\infty$. Noting $a\leq 1/A$ and $(1+R/a)/A>R$, there exists $t^*$ such that $R(t^*)=(1+R/a)/A$. 

\begin{lem}\label{lem:speed}
Let point $(0,y_{\sigma}(t))=\Gamma_{\sigma}(t)\cap\{(x,y)\mid x=0\}$, $t<T_{\sigma}$. There exists $\sigma_1$(indeed $\sigma_1$ in this lemma is the one we want to choose in Proposition \ref{prop:sigmalarge}) such that for all $\sigma>\sigma_1$ there holds that if $t^*<T_{\sigma}$,
$$
y_{\sigma}(t)>K,\ t<t^*.
$$ 
\end{lem}

We use the Grim reaper to prove this lemma.

\begin{proof}
Seeing that Grim reaper given by Lemma \ref{lem:grimreaper}
$$
G(x,t)=C-\frac{t}{b}+b\ln \cos\frac{x}{b}
$$
is a traveling wave with uniform speed $1/b$, then choose $C$ large enough such that $G(0,t)=C-t/b>C-t^*/b>K$, $t<t^*$.

We can choose $\sigma_1$ such that for all $\sigma>\sigma_1$, $\Gamma_{\sigma}\succ \gamma_{G}(0)$.
 
If $t^*<T_{\sigma}$, Lemma \ref{lem:grimreaper} implies that $\Gamma_{\sigma}(t)\succ \gamma_{G}(t)$, for $t<t^*$. This means that $y_{\sigma}(t)>K$, $t<t^*$.
\end{proof}

\begin{proof}[Proof of Proposition \ref{prop:sigmalarge}]
Choose $\sigma_1$ as in Lemma \ref{lem:speed}. 

{\bf Step 1.} For $\sigma>\sigma_1$, if $T_{\sigma}\leq t^*$($t^*$ is given in Lemma \ref{lem:speed}), this means that $T_{\sigma}<\infty$. By Lemma \ref{lem:sigulartime}, only the condition (4) in Lemma \ref{lem:sigulartime} can hold. Consequently, the result is true.

{\bf Step 2.} For $\sigma>\sigma_1$, if $t^*<T_{\sigma}$, by Lemma \ref{lem:speed}, $y_{\sigma}(t)>K$, $t<t^*$. Here we prove that there holds 
$$
\Gamma_{\sigma}(t)\succ\Gamma^*,\ t^*<t<T_{\sigma}.
$$ 

Let 
$$
\partial B(t)=\big\{(x,y)\mid (x-R)^2+\big(y-(1+R/a)\sqrt{1/A^2-a^2}\big)^2=R(t)^2\big\},
$$
where $R(t)$ is given by (\ref{eq:circle}). It is easy to see that $\partial B(t)$ evolves by $V=-\kappa+A$. 

Let $\Sigma(t)=\partial B(t)\cap\{(x,y)\mid x\leq0,y\geq0\}$. There exists $\delta$ satisfying $0<\delta<t^*$ such that 
$$
R(\delta)=\sqrt{R^2+(1+R/a)^2(1/A^2-a^2)}.
$$
Obviously, $\partial B(\delta)$ passes through the origin $(0,0)$. Seeing the Figure \ref{fig:circlesub1} and \ref{fig:circlesub2} and noting the choice of $t^*$, the boundary of $\Sigma(t)$ does not intersect $\Gamma_{\sigma}(t)$, $t<t^*$. By maximum principle, $\Sigma(t)$ can not intersect $\Gamma_{\sigma}(t)$ interior. Therefore, $\Gamma_{\sigma}(t)$ does not intersect $\Sigma(t)$, $t<t^*$. Here we omit the detail.

\begin{figure}[htbp]
	\begin{center}
            \includegraphics[height=6.0cm]{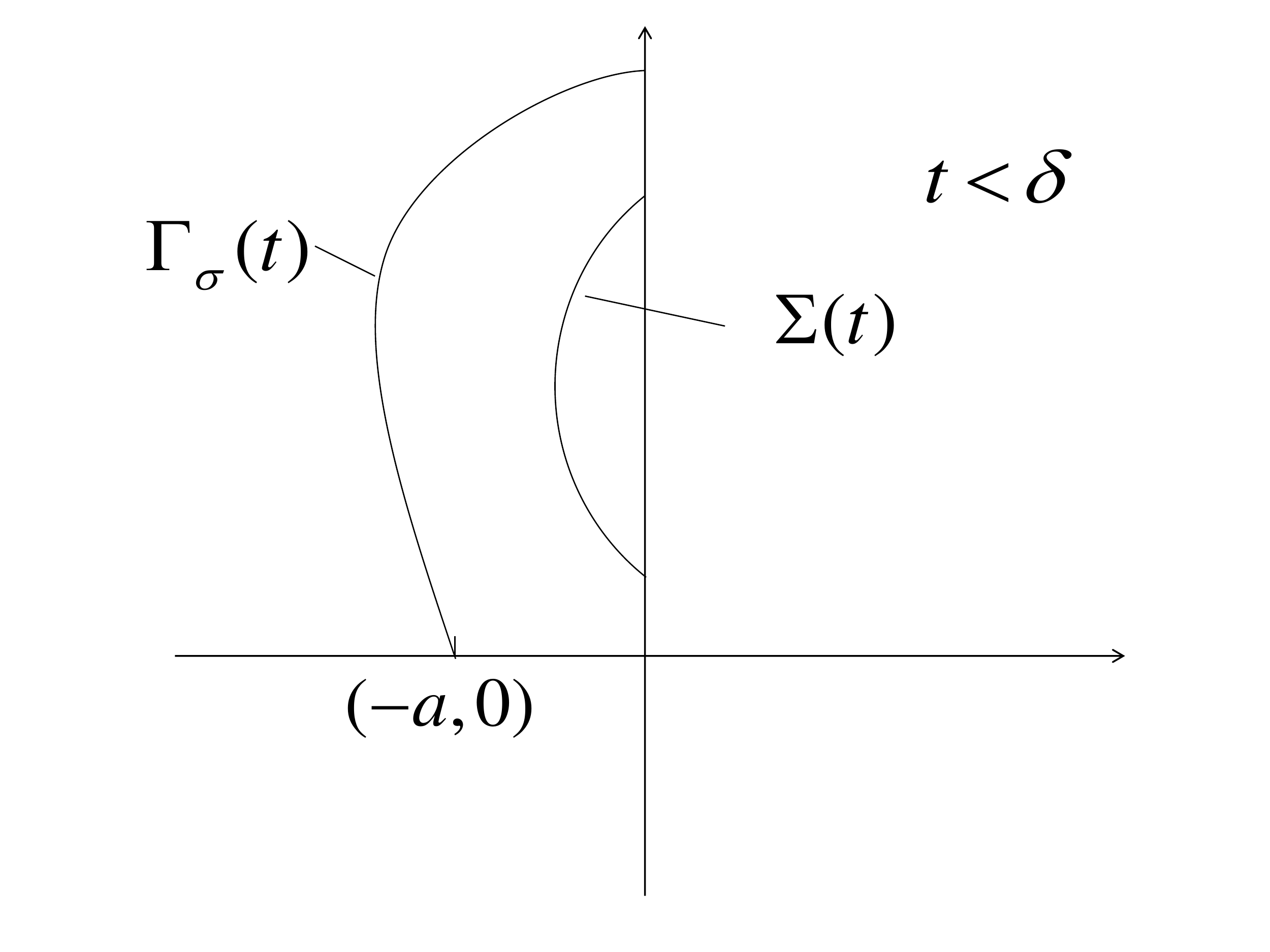}
		\vskip 0pt
		\caption{Proof of Proposition \ref{prop:sigmalarge}}
        \label{fig:circlesub1}
	\end{center}
\end{figure}
\begin{figure}[htbp]
	\begin{center}
            \includegraphics[height=6.0cm]{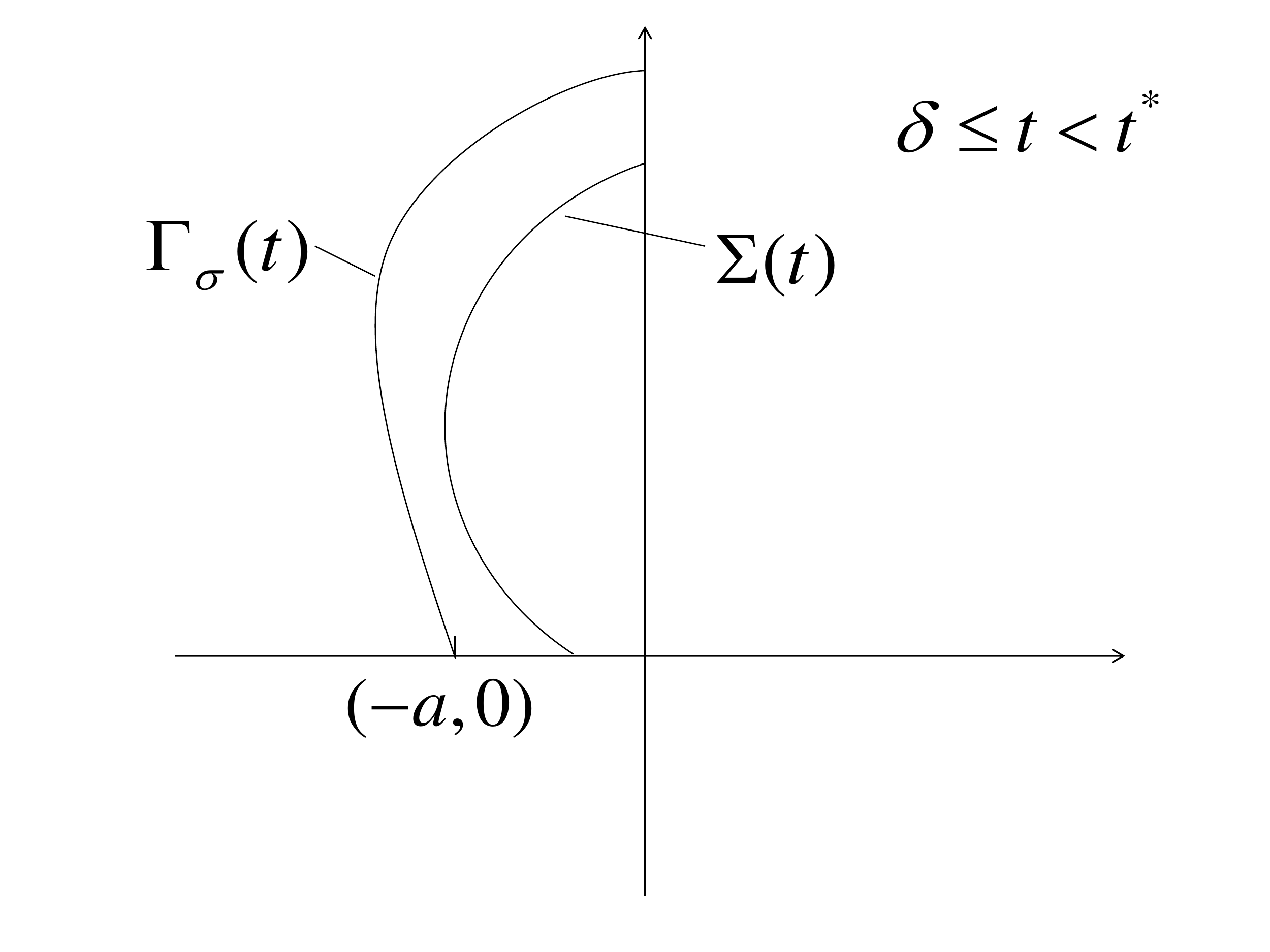}
		\vskip 0pt
		\caption{Proof of Proposition \ref{prop:sigmalarge}}
        \label{fig:circlesub2}
	\end{center}
\end{figure}

Seeing the Figure \ref{fig:circlesub3}, $\Sigma(t^*)$ intersects $\Gamma^*$ tangentially at $(-a,0)$. Since $\Gamma_{\sigma}(t)$ does not intersect $\Sigma(t)$ for $t<t^*$, then $\Gamma(t^*)\succ\Gamma^*$. Therefore, $\Gamma(t)\succ\Gamma^*$ for $t^*<t<T_{\sigma}$. Let $T_{\sigma}^*=t^*$, we complete the proof. 
\begin{figure}[htbp]
	\begin{center}
            \includegraphics[height=6.0cm]{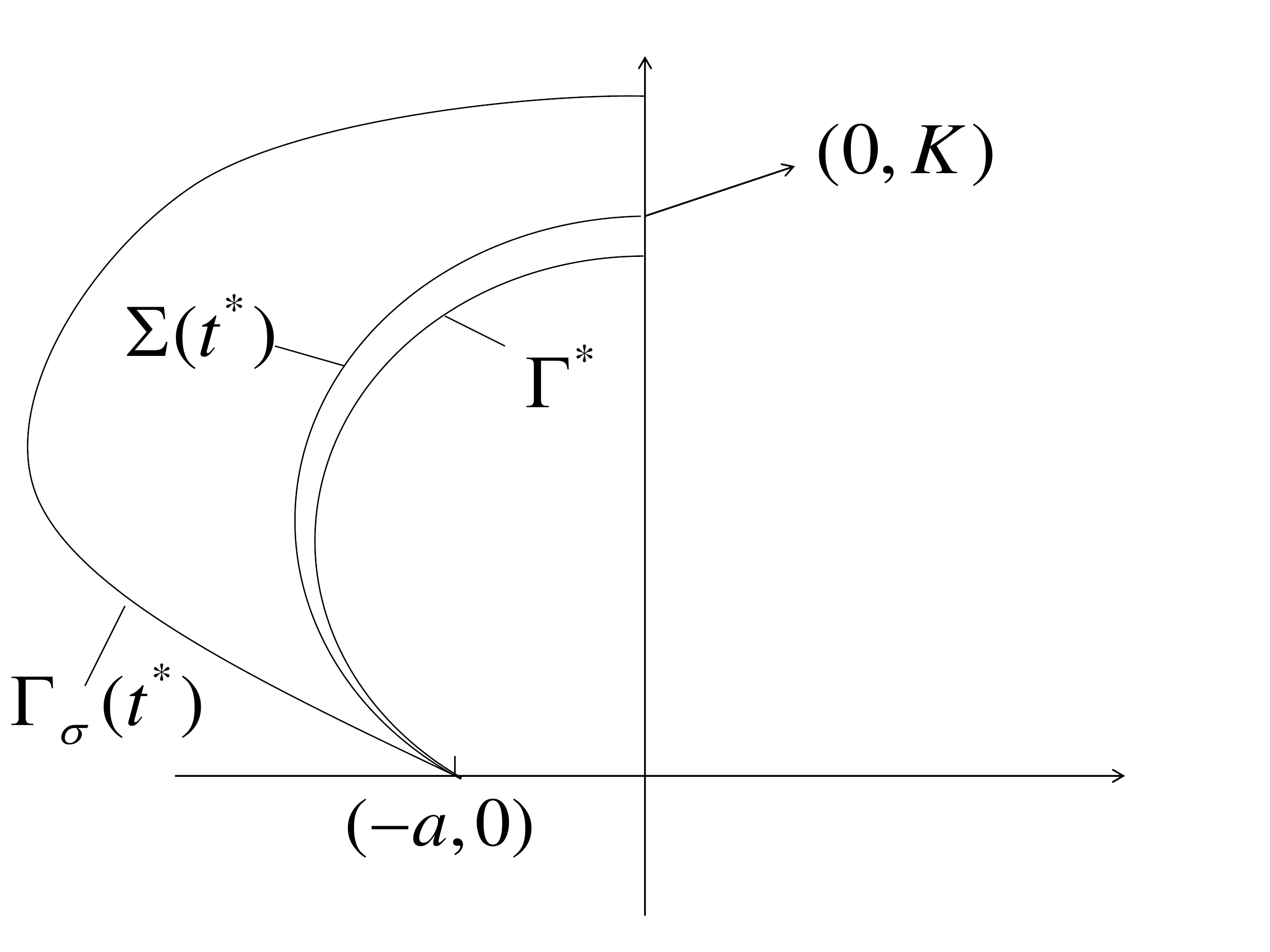}
		\vskip 0pt
		\caption{Proof of Proposition \ref{prop:sigmalarge}}
        \label{fig:circlesub3}
	\end{center}
\end{figure}

\end{proof}

\begin{prop}\label{prop:sigmasmall}
There exists $\sigma_2>0$ such that for all $\sigma<\sigma_2$, $T_{\sigma}=\infty$ and $\Gamma_{\sigma}(t)\rightarrow \Gamma_*$ in $C^1$, as $t\rightarrow\infty$.
\end{prop}

\begin{proof}
{\bf Step 1.} Upper bound.

 There exists $\sigma_2$ such that for all $\sigma<\sigma_2$, $\Gamma_*\succ\Gamma_{\sigma}$. Since $\Gamma_{\sigma}$ is represented by the graph of $\sigma\varphi$, then $\Gamma_{\sigma}(t)$ can locally be represented by the graph of some function $u_{\sigma}(x,t)$. Let $T_{\sigma}^g$ be the maximal time such that 
$$
\Gamma_{\sigma}(t)=\{(x,y)\mid y=u_{\sigma}(x,t)\},\ 0\leq t<T_{\sigma}^g.
$$
Therefore, $u_{\sigma}(x,t)$ satisfies
\begin{equation}\label{eq:drivinggraph}
\left\{
\begin{array}{lcl}
\dis{u_t=\frac{u_{xx}}{1+u_x^2}+A\sqrt{1+u_x^2}},\ -a<x<a,\ 0<t<T_{\sigma}^g,\\
u(-a,t)=u(a,t)=0,\ 0<t<T_{\sigma}^g,\\
u(x,0)=\sigma\varphi(x),\ -a\leq x\leq a.
\end{array}
\right.
\end{equation}

Since for all $\sigma<\sigma_2$ there holds $\sigma\varphi(x)\leq \sqrt{1/A^2-x^2}-\sqrt{1/A^2-a^2}$, $-a\leq x\leq a$, by comparison principle, 
\begin{equation}\label{eq:upb}
u_{\sigma}(x,t)< \sqrt{1/A^2-x^2}-\sqrt{1/A^2-a^2},\ -a< x< a, \ 0<t<T_{\sigma}^g.
\end{equation}

{\bf Step 2.} Lower bound and derivative estimate.

If $0\leq \sigma<\sigma_2$, by comparison principle, 
\begin{equation}\label{eq:downb1}
u_{\sigma}(x,t)>0,\ -a< x< a, \ 0<t<T_{\sigma}^g.
\end{equation}

Combining (\ref{eq:upb}) and (\ref{eq:downb1}), 
\begin{equation}\label{eq:db1}
-\frac{a}{\sqrt{1/A^2-a^2}}< u_{\sigma x}(a,t)<0\ \text{and}\ 0< u_{\sigma x}(-a,t)<\frac{a}{\sqrt{1/A^2-a^2}},\ 0<t<T_{\sigma}^g.
\end{equation}

Differentiating the first equation in (\ref{eq:drivinggraph}) by $x$ and combining boundary condition (\ref{eq:db1}), by maximum principle, 
\begin{equation}\label{eq:db2}
|u_{\sigma x}(x,t)|<\frac{a}{\sqrt{1/A^2-a^2}},\ -a\leq x\leq a, \ 0<t<T_{\sigma}^g.
\end{equation}

If $\sigma<0$, let $k>0$ satisfy $k:=\sigma\varphi^{\prime}(a)$. We denote function
$$
\underline{u}(x)=\max\{-k(x+a),k(x-a)\},\ -a\leq x\leq a.
$$
Obviously, $\underline{u}(x)\leq\sigma\varphi$, $-a\leq x\leq a$ and $\underline{u}$ is a sub-solution of (\ref{eq:drivinggraph}) in viscosity sense.

Therefore, by maximum principle, $u_{\sigma}(x,t)>\underline{u}(x)$, $-a<x<a$, $0<t<T_{\sigma}^g$. Combining (\ref{eq:upb}), we have
\begin{equation}\label{eq:db3}
|u_{\sigma x}(x,t)|<\max\{k,\frac{a}{\sqrt{1/A^2-a^2}}\},\ -a\leq x\leq a, \ 0<t<T_{\sigma}^g.
\end{equation}

Consequently, (\ref{eq:db2}) and (\ref{eq:db3}) imply that there exists $C_{\sigma}$ such that 
\begin{equation}\label{eq:dbc}
|u_{\sigma x}(x,t)|\leq C_{\sigma},\ -a\leq x\leq a, \ 0<t<T_{\sigma}^g.
\end{equation}

{\bf Step 3.} We prove the convergence in this step. 

By \cite{LSU}, for $\epsilon>0$, $u_{\sigma xx}(x,t)$ is bounded for all $-a\leq x\leq a,\ \epsilon\leq t<T_{\sigma}^g$. This means that $T_{\sigma}^g=\infty$. Therefore, by \cite{LSU} again, $u_{\sigma}(x,t)$, $u_{\sigma t}(x,t)$, $u_{\sigma tt}(x,t)$, $u_{\sigma x}(x,t)$, $u_{\sigma xx}(x,t)$ and $u_{\sigma xxx}(x,t)$ are all bounded for some constant $D_{\sigma}>0$, $-a\leq x\leq a,\ \epsilon\leq t<\infty$. For any sequence $t_n\rightarrow\infty$, there exist a subsequence $t_{n_j}$ and function $v(x,t)$ such that 
$$
u_{\sigma}(\cdot,\cdot+t_{n_j})\rightarrow v,\ \text{in}\ C^{2,1}([-a,a]\times[\epsilon,\infty)),
$$
as $j\rightarrow\infty$.

{\bf Step 4.} In this step, we introduce a Lyapunov function. 

Let 
$$
J[u]=\int_{-a}^a\sqrt{1+u_x^2}dx.
$$
If $u$ is a solution of (\ref{eq:drivinggraph}), we calculate
\begin{eqnarray*}
\frac{d}{dt}J[u]&=&\int_{-a}^a\frac{u_xu_{xt}}{\sqrt{1+u_x^2}}dx=-\int_{-a}^a \frac{u_tu_{xx}}{(1+u_x^2)^{3/2}}dx=-\int_{-a}^a\frac{(u_t)^2}{\sqrt{1+u_x^2}}dx +A\int_{-a}^au_tdx\\
&=&-\int_{-a}^a\frac{(u_t)^2}{\sqrt{1+u_x^2}}dx+A\frac{d}{dt}\int_{-a}^audx.
\end{eqnarray*}
Therefore, there hold
$$
J[u(\cdot,t)]\leq J[u(\cdot,\epsilon)]+A\int_{-a}^au(x,t)dx-A\int_{-a}^au(x,\epsilon)dx
$$
and
$$
\int_{\epsilon}^{\infty}\int_{-a}^a\frac{(u_t)^2}{\sqrt{1+u_x^2}}dxdt=A\lim\limits_{t\rightarrow\infty}\int_{-a}^au(x,t)dx-A\int_{-a}^au(x,\epsilon)dx+J[u(\cdot,\epsilon)]-\lim\limits_{t\rightarrow \infty}J[u(\cdot,t)].
$$

{\bf Step 5.} Using the Lyapunov function we complete the proof. 

For $u_{\sigma}$ given by above, $u_{\sigma}(x,t)$ is uniformly bounded for $-a\leq x\leq a$, $0<t<\infty$. Then the integral 
$$
|\int_{-a}^au_{\sigma}(x,t)dx|
$$
is bounded for $0<t<\infty$. Consequently, $J[u_{\sigma}(\cdot,t)]$ is bounded for $0<t<\infty$. Therefore, the integral
$$
\int_{\epsilon}^{\infty}\int_{-a}^a\frac{(u_{\sigma t})^2}{\sqrt{1+u_{\sigma x}^2}}dxdt
$$
is integrable. Then for all $s_0>0$,
$$
\int_{s_0}^{s_0+1}\int_{-a}^a\frac{(u_{\sigma t})^2}{\sqrt{1+u_{\sigma x}^2}}(x,t+t_{n_j})dxdt=\int_{s_0+t_{n_j}}^{s_0+1+t_{n_j}}\int_{-a}^a\frac{(u_{\sigma t})^2}{\sqrt{1+u_{\sigma x}^2}}(x,t)dxdt\rightarrow0
$$
as $j\rightarrow\infty$. Seeing $u_{\sigma}(\cdot,\cdot+t_{n_j})\rightarrow v$, in $\ C^{2,1}([-a,a]\times[\epsilon,\infty))$, as $j\rightarrow\infty$, we have
$$
\int_{s_0}^{s_0+1}\int_{-a}^a\frac{(v_{ t})^2}{\sqrt{1+v_x^2}}(x,t)dxdt=0.
$$
Then $v_t(x,t)=0$, for all $-a\leq x\leq a$, $s_0\leq t\leq s_0+1$. Seeing that the choice of $s_0$ is arbitrary, 
$$
v_t(x,t)=0,
$$
for all $-a\leq x\leq a$, $0<t<\infty$. Therefore, $v$ is independent on $t$ and is a stationary solution of (\ref{eq:drivinggraph}). Then 
$$
v=\sqrt{1/A^2-x^2}-\sqrt{1/A^2-a^2},\ -a\leq x\leq a.
$$ 

Here we get that for any sequence $t_n\rightarrow\infty$, there exists a subsequence $t_{n_j}$ such that 
$$
u_{\sigma}(\cdot,\cdot+t_{n_j})\rightarrow v,\ \text{in}\ C^{2,1}([-a,a]\times[\epsilon,\infty)),
$$
as $j\rightarrow\infty$. Consequently, 
$$
u_{\sigma}(\cdot,t)\rightarrow v,\ \text{in}\ C^{2}([-a,a]),
$$
as $t\rightarrow\infty$.

The proof of this proposition is completed.
\end{proof}

\section{Asymptotic behavior for the condition (3) in Lemma \ref{lem:sigulartime}}
Seeing Lemma \ref{lem:regular} and the proof of Lemma \ref{lem:sigulartime}, under the condition (3) in Lemma \ref{lem:sigulartime} we can assume there exists $\rho_{\sigma}$ such that
$$
\Gamma_{\sigma}(t)=\{(\rho_{\sigma}(\theta,t)\cos\theta,\rho_{\sigma}(\theta,t)\sin\theta)\mid0\leq \theta\leq \pi\},\ 0\leq t<\infty.
$$
Moreover $\rho_{\sigma}$ satisfies (\ref{eq:polarco}) for $T_{\sigma}=\infty$.

\begin{lem}\label{lem:ly}
Let $L_{\sigma}(t)$ be the length of $\Gamma_{\sigma}(t)$ and $S_{\sigma}(t)$ be the area of the domain surrounded by $\Gamma_{\sigma}(t)$ and $y=0$. Then,
\begin{equation}\label{eq:ly2}
\frac{d}{dt}L_{\sigma}(t)=-\int_{0}^{L_{\sigma}(t)}(\kappa-A)^2ds+A\frac{d}{dt}S_{\sigma}(t).
\end{equation}
\end{lem}
\begin{rem}
(1). Noting that under the condition (3) in Lemma \ref{lem:sigulartime}, $\Gamma_{\sigma}(t)$ located in $\{y\geq0\}$, the definition of $S_{\sigma}(t)$ is well-defined.

(2). The result of this lemma is a general condition for the Lyapunov function in the proof of Proposition \ref{prop:sigmasmall}.
\end{rem}

\begin{proof}
Seeing the calculation in \cite{Z2}, 
$$
\frac{d}{dt}L_{\sigma}(t)=\int_0^{L_{\sigma}(t)}(A\kappa-\kappa^2)ds.
$$
Recall $N$ being the unit downward normal vector. Therefore, 
\begin{eqnarray*}
\frac{d}{dt}L_{\sigma}(t)&=&-\int_0^{L_{\sigma}(t)}(\kappa-A)^2ds+\int_0^{L_{\sigma}(t)}(-A\kappa+A^2)ds=-\int_0^{L_{\sigma}(t)}(\kappa-A)^2ds\\
&+&A\int_0^{L_{\sigma}(t)}\frac{d}{dt}F(s,t)\cdot(-N)ds,
\end{eqnarray*}
where $F$ is the point on the curve $\Gamma_{\sigma}(t)$ and for convenience, we omit the subscript of $F_{\sigma}(s,t)$. Let 
$$
\gamma_{\sigma}(t)=\Gamma_{\sigma}(t)\cup\{(x,y)\mid y=0,-a\leq x\leq a\}.
$$
By Green's formula,
\begin{eqnarray*}
\frac{d}{dt}S_{\sigma}(t)=\frac{1}{2}\frac{d}{dt}\int_{\gamma_{\sigma}(t)}F(s,t)\cdot(-N)ds,
\end{eqnarray*}
where $F$ is the point on the curve $\gamma_{\sigma}(t)$ and $N$ is the unit inner normal vector.
Since the curve
$$
\{(x,y)\mid y=0,-a\leq x\leq a\}
$$
is independent on $t$, 
$$
\frac{1}{2}\frac{d}{dt}\int_{\gamma_{\sigma}(t)}F(s,t)\cdot(-N)ds=\frac{1}{2}\frac{d}{dt}\int_{\Gamma_{\sigma}(t)}F(s,t)\cdot(-N)ds=\frac{1}{2}\frac{d}{dt}\int_0^{L_{\sigma}(t)}F(s,t)\cdot(-N)ds.
$$
By calculation, 
\begin{eqnarray*}
\frac{1}{2}\frac{d}{dt}\int_0^{L_{\sigma}(t)}F(s,t)\cdot(-N)ds&=&\frac{1}{2}\int_0^{L_{\sigma}(t)}\frac{d}{dt}F(s,t)\cdot(-N)ds+\frac{1}{2}\int_0^{L_{\sigma}(t)}F(s,t)\cdot(-\frac{d}{dt}N)ds\\
&+&\frac{1}{2}\int_0^{L_{\sigma}(t)}F(s,t)\cdot(-N)(A\kappa-\kappa^2)ds.
\end{eqnarray*}
Seeing the calculation in \cite{Z2}, 
$$
\frac{d}{dt}N=-\frac{\partial\kappa}{\partial s}T,
$$
where $T=F_s$--the unit tangential vector, then
\begin{eqnarray*}
\frac{1}{2}\int_0^{L_{\sigma}(t)}F(s,t)\cdot(-\frac{d}{dt}N)ds=\frac{1}{2}\int_0^{L_{\sigma}(t)}F(s,t)\cdot(\frac{\partial\kappa}{\partial s}T)ds:=I.
\end{eqnarray*}
Seeing the symmetry of $\Gamma_{\sigma}(t)$ and $\kappa(0,t)=\kappa(L_{\sigma}(t),t)=A$, at the boundary,
$$
F(0,t)\cdot T(0,t)\kappa(0,t)=F(L_{\sigma}(t),t)\cdot T(L_{\sigma}(t),t)\kappa(L_{\sigma}(t),t)=AF(0,t)\cdot T(0,t).
$$
Integrating $I$ by parts, there holds
\begin{eqnarray*}
I&=&-\frac{1}{2}\int_0^{L_{\sigma}(t)}F_s\cdot T\kappa ds-\frac{1}{2}\int_0^{L_{\sigma}(t)}F\cdot T_{s}\kappa ds=-\frac{1}{2}\int_0^{L_{\sigma}(t)}\kappa ds\\
&-&\frac{1}{2}\int_0^{L_{\sigma}(t)}F\cdot F_{ss}\kappa ds=-\frac{1}{2}\int_0^{L_{\sigma}(t)}\kappa ds-\frac{1}{2}\int_0^{L_{\sigma}(t)}F\cdot N \kappa^2 ds.
\end{eqnarray*}
Therefore, 
\begin{eqnarray*}
\frac{1}{2}\frac{d}{dt}\int_0^{L_{\sigma}(t)}F(s,t)\cdot(-N)ds&=&\frac{1}{2}\int_0^{L_{\sigma}(t)}\frac{d}{dt}F(s,t)\cdot(-N)ds-\frac{1}{2}\int_0^{L_{\sigma}(t)}\kappa ds\\
&-&\frac{1}{2}\int_0^{L_{\sigma}(t)}AF(s,t)\cdot(\kappa N)ds=\frac{1}{2}\int_0^{L_{\sigma}(t)}\frac{d}{dt}F(s,t)\cdot(-N)ds\\
&-&\frac{1}{2}\int_0^{L_{\sigma}(t)}\kappa ds-\frac{1}{2}\int_0^{L_{\sigma}(t)}AF(s,t)\cdot(T_{s})ds\\
&=&\frac{1}{2}\int_0^{L_{\sigma}(t)}\frac{d}{dt}F(s,t)\cdot(-N)ds+\frac{1}{2}\int_0^{L_{\sigma}(t)}(A-\kappa)ds\\
&=&\int_0^{L_{\sigma}(t)}\frac{d}{dt}F(s,t)\cdot(-N)ds.
\end{eqnarray*}
In the last second equality, we use integral by parts. Therefore, 
$$
\int_0^{L_{\sigma}(t)}\frac{d}{dt}F(s,t)\cdot(-N)ds=\frac{d}{dt}S_{\sigma}(t).
$$
Consequently, (\ref{eq:ly2}) holds.
\end{proof}

\begin{lem}\label{lem:rhob}
 Under the condition (3) in Lemma \ref{lem:sigulartime}, $SGN(\Gamma_{\sigma}(t),\Gamma^*)=[-\ +\ -]$ for $t<\infty$, there exist $\rho_2>\rho_1>0$ such that 
$$
\rho_1<\rho_{\sigma}(\theta,t)<\rho_2,
$$
for $0\leq \theta\leq\pi$, $0<t<\infty$.
\end{lem}
\begin{proof}
First, we prove $\rho_{\sigma}<\rho_2$. We prove this by contradiction, assuming $\rho_{\sigma}$ is not bounded from above.

If $\rho_{\sigma}(\pi/2,t)$ is bounded for all $t$, we can easily prove that there exist some $0<\theta_0<\pi/2$ and $t_0$ such that $\kappa_{\sigma}(\theta_0,t_0)\leq 0$. This contradicts to that $\kappa_{\sigma}(\theta,t)>0$, for all $0<\theta<\pi$, $t<\infty$.

Therefore, $\rho_{\sigma}(\pi/2,t)$ is not bounded. Assume for some $t_0$, $\rho_{\sigma}(\pi/2,t_0)$ is large enough. We can use the Grim reaper argument as in Proposition \ref{prop:sigmalarge} to prove that $\Gamma_{\sigma}(t)\succ\Gamma^*$ in finite time. This contradicts to that $SGN(\Gamma_{\sigma}(t),\Gamma^*)=[-\ +\ -]$ for $t<\infty$. Therefore, there exists $\rho_2>0$ such that $\rho_{\sigma}(\theta,t)<\rho_2$,  for all $0<\theta<\pi$, $t<\infty$. 

On the other hand, we note that $\Gamma_{\sigma}\succ \Lambda_0=\{(x,y)\mid y=0,\ -a\leq x\leq a\}$, $0<t<\infty$. Then there exists $\rho_1>0$ such that $\rho_{\sigma}(\theta,t)>\rho_1$, $0<\theta<\pi$, $t<\infty$. 

We complete the proof.
\end{proof}
Here we give the asymptotic behavior under the condition (3) in Lemma \ref{lem:sigulartime}.
\begin{prop}\label{pro:asymup}
 Under the condition (3) in Lemma \ref{lem:sigulartime}, $SGN(\Gamma_{\sigma}(t),\Gamma^*)=[-\ +\ -]$ for $t<\infty$, 
$$
\Gamma_{\sigma}(t)\rightarrow \Gamma^*\ \text{in}\ C^1
$$
as $t\rightarrow\infty$.
\end{prop}
\begin{proof}
Since $\rho_1<\rho_{\sigma}<\rho_2$ and $\rho_{\sigma}$ satisfies (\ref{eq:polarco}), then there exists $\epsilon>0$ such that $\rho_{\sigma t}$, $\rho_{\sigma tt}$, $\rho_{\sigma\theta}$, $\rho_{\sigma\theta\theta}$ and $\rho_{\sigma\theta\theta\theta}$ are bounded for $0\leq \theta\leq \pi$, $\epsilon \leq t<\infty$. Therefore for any $t_n\rightarrow\infty$, there exist a subsequence $t_{n_j}$ and a function $r(\theta,t)$ such that
$$
\rho_{\sigma}(\cdot,\cdot+t_{n_j})\rightarrow r\ \text{in}\ C^{2,1}([0,\pi]\times[\epsilon,\infty)),
$$
as $j\rightarrow\infty$. Let 
$$
\gamma_{r}(t)=\{(x,y)\mid x=r(\theta,t)\cos\theta,\ y=r(\theta,t)\sin\theta,\ 0\leq\theta\leq \pi\},
$$
and the curvature $\kappa_{\sigma}(\cdot,t)$ be the curvature on $\Gamma_{\sigma}(t)$, $\kappa_{r}(\cdot,t)$ be the curvature on $\gamma_{r}(t)$. Therefore, $\kappa_{\sigma}(\cdot,\cdot+t_{n_j})\rightarrow \kappa_{r}$ in $C([0,\pi]\times[\epsilon,\infty))$. Obviously, the length $L_{\sigma}(t)$ and the area $S_{\sigma}(t)$ are bounded. Using the same argument in Proposition \ref{prop:sigmasmall} and the Lyapunov function in Lemma \ref{lem:ly}, we can deduce that $\kappa_{r}\equiv A$. Consequently, $r_t(\theta,t)=0$ for all $0\leq\theta\leq\pi$ and $t>0$. Seeing the curvature of $\gamma_{r}$ is a positive constant, $\gamma_{r}$ can only be a part of circle with radius $1/A$. Seeing $r(0,t)=-a$ and $r(\pi,t)=a$, then $\gamma_{r}=\Gamma^*$ or $\gamma_{r}=\Gamma_*$. But for all $t>0$,
$$
\Gamma_{\sigma}(t+t_{n_j})\rightarrow \gamma_{r}\ \text{in}\ C^1(\text{indeed,\ the\ convergence\ can\ be\ proved\ in\ $C^2$}),
$$
as $j\rightarrow\infty$.
If $\gamma_{r}=\Gamma_*$, then for $t$ large enough, $\Gamma^*\succ\Gamma_{\sigma}(t)$. This yields a contradiction. Therefore, $\gamma_{r}=\Gamma^*$. Consequently,
$$
\Gamma_{\sigma}(t)\rightarrow\Gamma^*\ \text{in}\ C^1(\text{indeed,\ the\ convergence\ can\ be\ proved\ in\ $C^2$}),
$$
as $t\rightarrow\infty$. 

Here we complete the proof.
\end{proof}

\section{Proof of Theorem \ref{thm:category}}
\begin{lem}\label{lem:sigmasmallc}
The set 
$$
B_*=\{\sigma\in \mathbb{R}\mid\Gamma_{\sigma}(t)\rightarrow\Gamma_*\ \text{in}\ C^1,\ \text{as}\ t\rightarrow\infty\}
$$
is open and connect.
\end{lem}
\begin{proof}
Proposition \ref{prop:sigmasmall} implies that $(-\infty,0]\subset B_*\neq \emptyset$. Therefore we only consider $\sigma_1>0$ and $\sigma_1\in B_*$.

(1). We prove that, for all $\sigma<\sigma_1$, $\Gamma_{\sigma}(t)\rightarrow\Gamma_*\ \text{in}\ C^1$, as $t\rightarrow\infty$. 

Since $\Gamma_{\sigma_1}(t)\rightarrow\Gamma_*\ \text{in}\ C^1$, as $t\rightarrow\infty$, then there holds $\Gamma^*\succ\Gamma_{\sigma_1}(t)$ for $t$ large enough. By comparison principle, $\Gamma_{\sigma_1}(t)\succ\Gamma_{\sigma}(t)$, $t<T_{\sigma}$. These imply that the condition (4) in Lemma \ref{lem:sigulartime} can not hold. Therefore, $T_{\sigma}=\infty$. By the same argument in Proposition \ref{prop:sigmasmall}, we can prove 
$$
\Gamma_{\sigma}(t)\rightarrow\Gamma_*\ \text{in}\ C^1,
$$
as $t\rightarrow\infty$. Here we prove that $B_*$ is connect. 

(2). We are going to prove $B_*$ is open. We only need prove that there exists $\epsilon_0>0$, $(\sigma_1,\sigma_1+\epsilon_0)\subset B_*$. We divide this proof into two steps.

{\bf Step 1.} Let 
$$
\tau_0(\sigma)=\max\{\tau \mid\frac{d}{ds}F_{\sigma}(0,t)\cdot(0,1)>0, 0<t<\tau\},\ \text{for}\  \sigma>0. 
$$ 
By comparison principle, we can prove that $\tau_0(\sigma)$ is non-increasing with respect to $\sigma$. let $\tau^*=\sup\{\tau_0(\sigma)\mid\sigma>\sigma_1\}=\lim\limits_{\sigma\downarrow\sigma^1}\tau_0(\sigma)$. We claim that $\tau^*=\infty$. 

For all $t<\tau^*$, there exists $\delta_0$, for $\sigma\in(\sigma_1,\sigma_1+\delta_0)$, $\tau_0(\sigma)>t$. Therefore, for $\sigma\in(\sigma_1,\sigma_1+\delta_0)$, $\Gamma_{\sigma}(t)$ can be represented by polar coordinate and not become singular. By Lemma \ref{lem:localconti},
$$
\Gamma_{\sigma}(t)\rightarrow\Gamma_{\sigma_1}(t),
$$ 
as $\sigma\rightarrow\sigma_1$. Seeing that the condition (1) or (2) in Lemma \ref{lem:sigulartime} hold for $\Gamma_{\sigma_1}(t)$, we can prove that there exists $\delta>0$ such that 
$$
\frac{d}{ds}F_{\sigma_1}(0,t)\cdot(0,1)>\delta,\ t<\infty.
$$
Consequently,
$$
\lim\limits_{\sigma\downarrow\sigma^1}\frac{d}{ds}F_{\sigma}(0,t)\cdot(0,1)\geq\delta,\ t<\tau^*.
$$
Therefore, $\tau^*=\infty$.

{\bf Step 2.} We complete the proof.

We choose two curves $\gamma_1$ and $\gamma_2$ such that $\Gamma^*\succ\gamma_1\succ\Gamma_*\succ\gamma_2$ and the domain $V$ be the shuttle neighbourhood of $\Gamma_*$ satisfying $\partial V=\gamma_1\cap\gamma_2$. 

Since
$$
\Gamma_{\sigma_1}(t)\rightarrow\Gamma_*\ \text{in}\ C^1,
$$
as $t\rightarrow\infty$, for $t_0$ large enough $\Gamma_{\sigma_1}(t_0)\subset V$. Seeing the result in Step 1, for $\sigma$ close to $\sigma_1$, $\Gamma_{\sigma}(t_0)$ can be represented by polar coordinate and not become singular. By Lemma \ref{lem:localconti},
$$
\Gamma_{\sigma}(t_0)\rightarrow\Gamma_{\sigma_1}(t_0),\ \text{in}\ C^1,
$$ 
as $\sigma\rightarrow\sigma_1$. Then there exists $\epsilon_0$ for $\sigma\in(\sigma_1,\sigma_1+\epsilon_0)$, $\Gamma_{\sigma}(t_0)\subset V$. Using the Lyapunov function given by Lemma \ref{lem:ly} and the same argument in Proposition \ref{pro:asymup}, for all $\sigma\in(\sigma_1,\sigma_1+\epsilon_0)$,
$$
\Gamma_{\sigma}(t)\rightarrow\Gamma_*\ \text{in}\ C^1,
$$
as $t\rightarrow\infty$.

We complete the proof.
\end{proof}

\begin{lem}\label{lem:sigmalargec}
The set 
$$
B^*=\{\sigma\in \mathbb{R}\mid \text{there\ exists}\ T_{\sigma}^*>0\ \text{such\ that}\ \Gamma_{\sigma}(t)\succ\Gamma^*,\ T_{\sigma}^*<t<T_{\sigma}\}
$$
is open and connect.
\end{lem}

\begin{proof}
Proposition \ref{prop:sigmalarge} and \ref{prop:sigmasmall} show that $B^*\subset (0,\infty)$ is not empty. We consider $\sigma_2>0$ and $\sigma_2\in B^*$. Then there exists $T_{\sigma_2}^*>0$ such that
$$
\Gamma_{\sigma_2}(t)\succ\Gamma^*,\ T_{\sigma_2}^*<t<T_{\sigma_2}.
$$ 

(1). We prove $(\sigma_2,\infty)\subset B^*$.

For $\sigma>\sigma_2$, if $T_{\sigma}<\infty$, then only the condition (4) in Lemma \ref{lem:sigulartime} can hold. The result is true. If $T_{\sigma}=\infty$, then by comparison principle, 
$$
\Gamma_{\sigma}(t)\succ\Gamma_{\sigma_2}(t)\succ\Gamma^*,\ T_{\sigma_2}^*<t<T_{\sigma_2}.
$$
Here we complete the proof that $B^*$ is connect.

(2). We prove $B^*$ is open. We only need prove that there exists $\epsilon_0>0$ such that $(\sigma_2-\epsilon_0,\sigma_2)\subset B^*$.

We can choose $t_0$ such that $\Gamma_{\sigma_2}(t_0)\succ \Gamma^*$ and 
$$
\frac{d}{ds}F_{\sigma_2}(0,t)\cdot(0,1)>0,\ 0<t\leq t_0.
$$
By Lemma \ref{lem:regular} and comparison principle, it is easy to see that for all $0<\sigma<\sigma_2$, $T_{\sigma}>t_0$. For $\sigma$ close to $\sigma_2$, $\Gamma_{\sigma}(t)$ can be represented by polar coordinate for $0<t\leq t_0$. By Lemma \ref{lem:localconti},
$$
\Gamma_{\sigma}(t_0)\rightarrow\Gamma_{\sigma_2}(t_0),\ \text{in}\ C^1
$$
as $\sigma\rightarrow\sigma_2$. Therefore, there exists $\epsilon_0>0$ such that for all $\sigma\in(\sigma_2-\epsilon_0,\sigma_2)$, $\Gamma_{\sigma}(t_0)\succ\Gamma^*$. By comparison principle, we can get the result easily.

We complete the proof.
\end{proof}

\begin{cor}\label{cor:connect}
There exist $0<\sigma_*\leq \sigma^*$ such that 
$$
B^*=(\sigma^*,\infty)
$$
and 
$$
B_*=(-\infty,\sigma_*).
$$
\end{cor}
\begin{proof}
Let $\sigma^*=\inf B^*$ and $\sigma_*=\inf B_*$. Obviously, $\sigma_*\leq\sigma^*$.

Lemma \ref{lem:sigmasmallc} and \ref{lem:sigmalargec} imply that
$$
B^*=(\sigma^*,\infty)
$$
and 
$$
B_*=(-\infty,\sigma_*).
$$
Proposition \ref{prop:sigmasmall} shows that $\sigma_*>0$. The proof is completed.
\end{proof}

\begin{prop}\label{pro:barrierfunction}
If $\Gamma_{\sigma_0}(t)\rightarrow\Gamma^*$ in $C^1$, for some $\sigma_0$, as $t\rightarrow\infty$, then $(\sigma_0,\infty)\subset B^*$.
\end{prop}
\begin{proof}
For $\sigma>\sigma_0$, if $T_{\sigma}<\infty$, then the result is true. In the following proof, we only consider $T_{\sigma}=\infty$. By comparison principle, $\Gamma_{\sigma}(t)\succ\Gamma_{\sigma_0}(t)$, $t>0$.

Since 
$$
\Gamma_{\sigma_0}(t)\rightarrow\Gamma^*\ \text{in}\ C^1,
$$
there exist $t_0$ and $\delta>0$ such that for all $t\geq t_0$, there hold 
\begin{equation}\label{eq:ylarz}
\frac{d}{ds}F_{\sigma_0}(0,t)\cdot(0,1)\geq \delta
\end{equation}
and
\begin{equation}\label{eq:xlesz}
\frac{d}{ds}F_{\sigma_0}(0,t)\cdot(1,0)\leq -\delta.
\end{equation}
Since $\Gamma_{\sigma}(t_0)\succ\Gamma_{\sigma_0}(t_0)$, then we can choose a small positive constant $c$ such that 
$$
\Gamma_{\sigma}(t_0)\succ\Gamma^{c}(t_0),
$$
where
$$
\Gamma^{c}(t)=\{(x,y+c)\mid (x,y)\in\Gamma_{\sigma_0}(t)\}\cup \{(-a,y)\mid 0\leq y\leq c\}\cup \{(a,y)\mid 0\leq y\leq c\},\ t>0.
$$
\begin{figure}[htbp]
	\begin{center}
            \includegraphics[height=6.0cm]{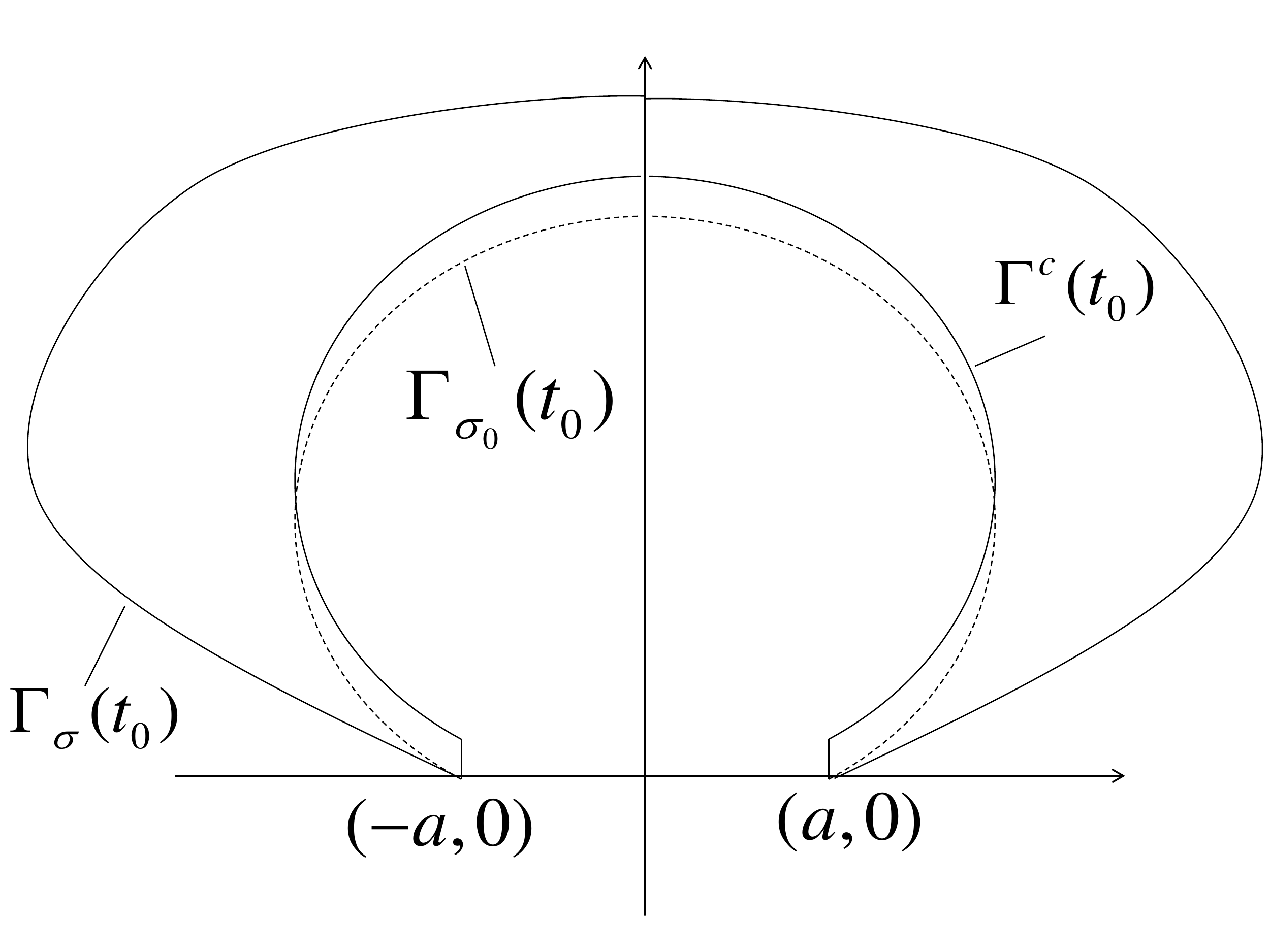}
		\vskip 0pt
		\caption{Construction of $\Gamma_c(t)$}
        \label{fig:barrier}
	\end{center}
\end{figure}

We claim that $\Gamma^c(t)$ is a sub-solution for $t\geq t_0$. Indeed, the part $\Gamma^c(t)\cap\{y> c\}$ is a translation of $\Gamma_{\sigma_0}(t)$. $\Gamma^c(t)\cap\{y> c\}$ satisfies (\ref{eq:meancurpara}). Since the part $\Gamma^c(t)\cap\{y< c\}$ consists of two straight lines, then the part is a sub-solution of (\ref{eq:meancurpara}). Next at the points $\{(-a,c),(a,c)\}=\Gamma^c(t)\cap\{y=c\}$, seeing (\ref{eq:ylarz}), (\ref{eq:xlesz}) for any smooth flow $S(t)$ can not touch $(-a,c)$ or $(a,c)$ above only once. Therefore $\Gamma^c(t)$ is a sub-solution of (\ref{eq:meancurpara}) and (\ref{eq:fixbound}) in the sense of Definition \ref{def:subsup}. 

By $\Gamma_{\sigma}(t_0)\succ\Gamma^c(t_0)$ and $\Gamma^c(t)$ being sub-solution for $t>t_0$, 
\begin{equation}\label{eq:largesub}
\Gamma_{\sigma}(t)\succ\Gamma^c(t),\ t>t_0. 
\end{equation}
Noting that $\Gamma^c(t)\rightarrow \Gamma^{*c}$ in $C$, as $t\rightarrow\infty$, where
$$
\Gamma^{*c}=\{(x,y+c)\mid (x,y)\in\Gamma^*\}\cup \{(-a,y)\mid 0\leq y\leq c\}\cup \{(a,y)\mid 0\leq y\leq c\}.
$$
If $\Gamma_{\sigma}(t)$ satisfies the condition (1) or (2) or (3) in Lemma \ref{lem:sigulartime}, seeing the proof of Proposition \ref{pro:asymup}, $\Gamma_{\sigma}(t)\rightarrow \Gamma^*$ or $\Gamma_{\sigma}(t)\rightarrow \Gamma_*$, as $t\rightarrow\infty$. 

Combining (\ref{eq:largesub}), $\Gamma^*\succeq \Gamma^{*c}$ or $\Gamma_*\succeq \Gamma^{*c}$. But all of these conditions are impossible. Therefore, only the condition (4) in Lemma \ref{lem:sigulartime} holds.

The proof is completed.
\end{proof}

\begin{proof}[Proof of Theorem \ref{thm:category}]
Let $\sigma_*$ and $\sigma^*$ be given by Corollary \ref{cor:connect}. 

Seeing the definition of $\sigma_*$, $\sigma_*\notin B^*$ and $\sigma_*\notin B_*$. Therefore, $\Gamma_{\sigma_*}(t)$ only satisfies the condition (3) in Lemma \ref{lem:sigulartime}. The result in Section 5 shows that $\Gamma_{\sigma_*}(t)\rightarrow \Gamma^*$, as $t\rightarrow\infty$.

By Proposition \ref{pro:barrierfunction}, $(\sigma_*,\infty)=B^*$. Consequently, $\sigma_*=\sigma^*$. 

The proof of Theorem \ref{thm:category} is completed.
\end{proof}
{\bf Acknowledge.}
The author is grateful to Professor Matano Hiroshi for his inspiring suggestion about Grim reaper argument. He is also grateful to Professor Giga Yoshikazu for letting me know several related useful papers.

{\bf Conflict of interest.} 
We declare that we do not have any commercial or associative interest that represents a conflict of interest in connection with the work submitted.

\end{document}